\documentclass{amsart}
\usepackage{mathrsfs}
\usepackage[all]{xy}

\usepackage{mathtools}
\usepackage{mathbbol}
\usepackage{wasysym}
\usepackage{amssymb}
\usepackage{MnSymbol}
\usepackage{comment}
\usepackage{enumerate}

\usepackage{ushort}

\usepackage{changepage}
\usepackage{algorithm}
\usepackage[noend]{algpseudocode}
\makeatletter
\def\BState{\State\hskip-\ALG@thistlm}
\makeatother

\usepackage{ifpdf}
\ifpdf 
  \usepackage[pdftex]{graphicx}
  \DeclareGraphicsExtensions{.pdf,.png,.jpg,.jpeg,.mps}
  \usepackage{pgf}
\else 
  \usepackage{graphicx}
  \DeclareGraphicsExtensions{.eps,.bmp}
  \usepackage{pgf}
  \usepackage{pstricks}
\fi
\usepackage{epic,bez123}
\usepackage{wrapfig}

\newtheorem{thm}{Theorem}[section]
\newtheorem{prop}[thm]{Proposition}
\newtheorem{cor}[thm]{Corollary}
\newtheorem{lem}[thm]{Lemma}
\newtheorem{conj}[thm]{Conjecture}
\newtheorem*{claim}{Claim}

\newtheorem{mainthm}{}

\newtheorem*{mainthmB}{Main Theorem B}

\newtheorem*{conv}{Convention}

\theoremstyle{definition}
\newtheorem{defn}[thm]{Definition}

\theoremstyle{remark}
\newtheorem*{quest}{Question}
\newtheorem*{rem}{Remark}

\newtheorem*{examples}{Examples}

\newtheorem*{ack}{Acknowledgment}

\newcommand{\isom}{\mbox{Isom}} 
 
\newcommand{\pGf}{\partial_\lambda{G}}

\newcommand{\pX}{\partial{\mathrm X}}
\newcommand{\pXf}{\partial_\lambda{X}}

\newcommand{\PS}[1]{\mathcal P_{#1}(s,o)}

\newcommand{\e}[1]{\omega(#1)}
\newcommand{\ec}[1]{\omega^c(#1)}

\newcommand{\ax}{\mathrm{Ax}}

\newcommand{\diam }[1]{{\textbf{diam}\big(#1\big)}}

\newcommand{\proj}{\textbf{d}^\pi}

\newcommand{\len }{\ell}

\usepackage[bookmarks=true, pdfauthor={YANG Wenyuan}]{hyperref}

\newtheoremstyle{query}%
{}{}
{\color{red}}
{}
{\sffamily\bfseries}{:}{12pt}
{}
\theoremstyle{query}

\calclayout
\begin{document}

\title[Genericity of contracting elements in groups]{Genericity of contracting elements in groups}

\author{Wen-yuan Yang}

\address{Beijing International Center for Mathematical Research (BICMR), Beijing University, No. 5 Yiheyuan Road, Haidian District, Beijing, China}
 
\email{yabziz@gmail.com}
\thanks{}



\date{19 July, 2017}

\dedicatory{}


\begin{abstract}
In this paper,   we establish that, for statistically convex-cocompact   actions, contracting elements are exponentially generic in counting measure.   Among others, the following exponential genericity results are obtained as corollaries for the set of hyperbolic elements in relatively hyperbolic groups, the set of rank-1 elements in CAT(0) groups, and the set of pseudo-Anosov elements  in mapping class groups.
 
Regarding a proper action, the set of non-contracting elements is proven to be growth-negligible. In particular, for mapping class groups,   the set of  pseudo-Anosov elements is generic  in a sufficiently large subgroup,  provided that the subgroup has purely exponential growth.  
By Roblin's work, we obtain that the set of hyperbolic elements  is generic in any discrete group action on CAT(-1) space with finite BMS measure.   

Applications to the number of conjugacy classes of non-contracting elements are given for non-rank-1 geodesics in CAT(0) groups with rank-1 elements.
\end{abstract}

\maketitle

\setcounter{tocdepth}{1} \tableofcontents

\section{Introduction}
\subsection{Main results on genericity}\label{SSgenericity}

In recent years, the notion of a contracting element is receiving a great deal of interests in studying various classes of groups with negative curvature. The prototype of this notion is  a hyperbolic isometry on hyperbolic spaces, but more interesting examples are furnished by the following:
\begin{itemize}
\item
hyperbolic elements in relatively hyperbolic groups, cf. \cite{GePo4}, \cite{GePo2};
\item
rank-1 elements in CAT(0) groups, cf. \cite{Ballmann}, \cite{BF2};
\item
certain infinite order elements in graphical small cancellation groups, cf. \cite{ACGH};
\item
pseudo-Anosov elements in mapping class groups, cf. \cite{Minsky}.
\end{itemize}

Usually, the existence of a contracting element represents a   situation where a certain negative curvature along geodesics exists in ambient spaces although with which non-negatively curved parts could coexist. One of the findings of the present study is that under natural assumptions, the negatively curved portion dominates the remaining part, once one contracting element is supplied. This fits in the rapidly developping  research scheme where the statistical and random properties are studied in counting measure (compared with harmonic measure in random walk). To be precise, we are aiming to address the following question in the present paper:  
\begin{quest}
Suppose that a countable group $G$ admits a proper and  isometric action on a proper 
geodesic metric space $(\mathrm Y, d)$ with a contracting element. Fix a basepoint $o\in \mathrm Y$. Denote $$N(o, n):=\{g\in G: d(o, go)\le n\}.$$
Is the set of contracting elements  \textit{generic} in counting measure, i.e.   
$$
\frac{\sharp \{g\in N(o, n): g \text{ is contracting}\}}{\sharp N(o, n)}\to 1,
$$
as $n\to \infty$? It is called \textit{exponentially generic} if the rate of convergence happen exponentially fast.
\end{quest}

Let us clarify the question by introducing a few more notions. The \textit{critical exponent} $\e \Gamma$ for a {subset} $\Gamma \subset G$ is defined by
$$
\e \Gamma = \limsup\limits_{n \to \infty} \frac{\log \sharp(N(o, n)\cap \Gamma)}{n},
$$
which is independent of the choice of $o \in \mathrm Y$.  A subset $X$ in $G$ is called \textit{growth-tight} if $\e X<\e G$;   \textit{growth-negligible} if $\sharp X  = o\big(\exp(\e G n)\big)$, where $o$ is the Landau little-o notation.
    
Assuming that $G$ has the following \textit{purely exponential growth} property:
$$
\sharp N(o, n) \asymp \exp(\e G n),
$$
the above question is rephrased as saying that whether the set of non-contracting elements is  growth-negligible/growth-tight.

 In \cite{YANG10}, we investigated the asymptotic geometry of  a class of actions called \textit{statistically convex-cocompact actions} (SCC) (cf. \textsection \ref{SSSCC}). Among other things, we proved that SCC actions have purely exponential growth. A class of \textit{barrier-free} sets were introduced and proved to be growth-tight for SCC actions, and growth-negligible for a general proper action.  These results therefore constitute the basis for the present study.

 The group $G$ is always assumed to be  \textit{non-elementary}: there is no finite-index cyclic subgroup.    The main theorem of this paper is the following.

\begin{mainthm}\label{GenericThm}
Suppose that a non-elementary group $G$ admits a proper (resp. SCC) action on a geodesic metric space $(\mathrm Y, d)$ with a contracting element. Then the set of non-contracting elements in $G$ has growth-negligible (resp. growth-tight). 

In particular, for SCC actions, contracting elements are exponentially generic.
\end{mainthm}

The distinction between SCC actions and proper actions is subtle in this study, in that for proper actions even with purely exponential growth, our methods only allow to obtain the genericity of contracting elements, while SCC actions get the exponential genericity in one shot. It is a natural problem to determine when the exponential genericity holds as well for a proper action.

Consider first the applications to the weak form of genericity, before turning to the exponential genericity.  
\begin{cor}
Assume that a non-elementary group $G$ admits a proper  action on a geodesic metric space $(\mathrm Y, d)$ with a contracting element. If $G$ has purely exponential growth, then the set of contracting elements is generic.
\end{cor}

We now explain two implications in some specific classes of groups. In the context of mapping class groups, a sufficiently large subgroup was first studied  by McCarthy and Papadopoulos \cite{McPapa}, as an analog of non-elementary subgroups of Kleinian groups. By definition, a \textit{sufficiently large} subgroup is  one with at least two independent pseudo-Anosov elements. We refer the reader to \cite{McPapa} for a detailed discussion. The interesting examples include convex-cocompact  subgroups  \cite{FarbMosher}, the handlebody group, among many others. The following result reduces the genericity question to the problem whether the subgroup has purely exponential growth.

\begin{thm}[$\mathbb{Mod}$]\label{ModGeneric}
In mapping class groups, the set of non-pseudo-Anosov elements in a sufficiently large subgroup $\Gamma$ has  growth-negligible. In particular,  if $\Gamma$ has purely exponential growth, then  the set of pseudo-Anosov elements is generic. 
\end{thm}

An obvious instance with purely exponential growth is a subclass of sufficiently large subgroups,  whose action on $\mathrm Y$ is itself SCC. This was studied in \cite{YANG10} as SCC subgroups, a dynamical generalization of convex-cocompact subgroups. We emphasize that there  exists indeed non-convex-cocompact non-free and free subgroups admiting SCC actions on Teichm\"{u}ller space. Hence, for these groups, the set of contracting elements is (exponentially) generic.  

In terms of Theorem \ref{ModGeneric}, it would be desirable to investigate which sufficiently large subgroups have  
purely exponential growth. The similar problem has been completely answered in the setting of Riemannian manifolds with variable negative curvature by work of T. Robin \cite{Roblin}. In fact, his results were obtained in a more general setting. For any proper action on a CAT(-1) space, Roblin showed that the growth of the action is of purely exponential if and only if the corresponding Bowen-Margulis-Sullivan measure is finite on the geodesic flow.  From this, we obtain the following application to discrete groups on a CAT(-1) space.
\begin{thm}[$\mathbb{CAT(-1)}$]\label{CAT1Generic}
Suppose that $G$ acts properly on a CAT(-1) space such that the BMS measure is finite. Then the set of hyperbolic elements is generic. 
\end{thm}

As suggested in Riemannian case, it seems interesting to develop an analogue of Roblin's result for sufficiently large subgroups in mapping class groups. 
 
We are now turning to the class of SCC actions, which admits stronger consequence by \ref{GenericThm}.  This class of groups encompasses many interesting clases of groups, listed as in \cite{YANG10}. Again, we first  consider the instance of mapping class groups. They are known to act on Teichm\"{u}ller space by SCC actions a result of by Eskin, Mirzakhani and Rafi \cite[Theorem 1.7]{EMR}, as observed in \cite[Section 10]{ACTao}. This is the starting point of our approach to mapping class groups, and modulo this input, our   argument is completely general and is working for any SCC action on a metric space. 

 Using dynamical methods, Maher \cite{Maher} has proved the genericity of pseudo-Anosov elements for the Teichm\"{u}ller metric. \footnote{The author is indebted to I. Gehktman and S. Taylor for telling him of Maher's result at MSRI in Aug. 2016 when this manuscript was in its final stage of preparation.}    Strenghening this result, \ref{GenericThm} allows to obtain the exponential genericity of peusdo-Anosov elements.

\begin{thm}[$\mathbb{Mod}$]\label{ModGeneric2}
In mapping class groups, the set of peusdo-Anosov elements   is exponentially generic in Teichm\"{u}ller metric. More generally,  given any convex-cocompact subgroup $\Gamma$, the set of  pseudo-Anosov elements not conjugated into $\Gamma$ are exponentially generic.
\end{thm}

After the first version was posted in Arxiv \cite{YANG10}, we learnt that this result was independently obtained by Spencer Dowdall in joint work with Howard  Masur in a conference ``Geometry of mapping class groups and Out($F_n$)'', October 25 - October 28, 2016.

\begin{rem}
In fact, Maher proved in \cite{Maher} that the set of  elements  (containing possibly pseudo-Anosov ones) with bounded translation length on curve graphs is negligible. Our last statement strenghens it as well to the exponential negligiblilty, and even more: since all conjugates of a convex-cocompact subgroup can have arbitrary large  translation length on curve graphs. See Lemma \ref{FMtight2}.
\end{rem}


Our study also shelds some light on a conjecture of Farb \cite{Farb2}, which predicts genericity for a class of combinatorial metrics (in contract with Teichm\"{u}ller metric in Maher's theorem).  
\begin{conj} 
The set of pseudo-Anosov elements is generic with respect to the word metric. 
\end{conj}
Since any group acts properly and cocompactly on its Cayley graph, the action of mapping class groups on Cayley graphs is SCC.  Farb's conjecture can be reduced by \ref{GenericThm} to the following: 
\begin{conj}\label{FarbConj}
Mapping class groups act on their Cayley graphs with contracting elements. 
\end{conj}
\begin{rem} 
Note that a contracting element in any proper action of $\mathbb {Mod}$ on a geodesic metric space has to be pseudo-Anosov, since a reducible element  is of infinite index in its centralizer, while a contracting element is of finite index. So, if this conjecture is true, it will imply the stronger version of Farb's conjecture that pseudo-Anosov elements is exponentially generic.
\end{rem}

To give more results on genericity, we mention  some applications of \ref{GenericThm} to another two classes of relatively hyperbolic groups and CAT(0) groups with rank-1 elements.

The class of relatively hyperbolic groups, introduced by Gromov \cite{Gro} as     generalization of hyperbolic groups, has been developped by many people \cite{Farb}, \cite{Bow1}, \cite{Osin}, \cite{Dah}, \cite{DruSapir}, \cite{Ge1}  and so on. In last twenty years, their intensive study achieves a huge success and inspires as well many research works on other classes of groups with negative curvature. For instance, the recent work of Dahmani, Guirardel and Osin \cite{DGO} on hyperbolically embeded subgroups with applications to mapping class groups.

It is widely believed that hyperbolic elements are generic in a relatively hyperbolic group, which by defintion are infinite order elements not conjugated into maximal parabolic subgroups.  Our next result confirms this for the action on Cayley graphs with respect to word metric. 
\begin{thm}[$\mathbb{RelHyp}$] \label{RelHypGeneric}
The set of hyperbolic elements in a relatively hyperbolic group is exponentially generic with respect to the word metric. 
\end{thm}
\begin{rem}
Note that this is not a direct consequence  of   \ref{GenericThm}, since contracting elements might be parabolic. The proof indeed follows from a more general result in \textsection \ref{Section4}.

We remark that this same conclusion holds for the cusp-uniform action with parabolic gap property, or a more general property introduced by Dal'bo, Otal and Peign\'e \cite{DOP}. Since, under these properties, the corresponding action has purely exponential growth, see \cite{YANG8}. 
\end{rem}

It is well-known that  a free product  of any two groups (or generally,  a graph of groups with finite edge groups) is hyperbolic relative to factors, by an equivalent definition in \cite{Bow1}. We thus obtain the following corollary.
\begin{cor}[Free product]
The set of elements in a free product of two groups not conjugated into both factors is exponentially generic with respect to the word metric, if the free product is not elementary (i.e., $\ncong  \mathbb Z_2 \star \mathbb Z_2.$). 
\end{cor}

Lastly, we consider the implications for the class of CAT(0) groups with rank-1 elements, which admits a geometric (and thus SCC) action with a contracting element.  There are two important subclasses, right-angled Artin groups (RAAG) and right-angled Coxeter groups (RACG), which are recieving a great deal of interests in last years. For stating the next result, we shall  explain which RAAGs and RACGs contain contracting elements.

It is well-known that an RAAG acts properly and cocompactly on a non-positively curved  cube complex called the \textit{Salvetti complex}.  The defining graph of a RAAG is a join if and only if the RAAG is a direct product of non-trivial groups.  In \cite[Theorem 5.2]{BehC}, Behrstock and Charney proved that any subgroup of an  RAAG $G$   that is not conjugated into a \textit{join subgroup} (i.e., obtained from a join subgraph)  contains a  contracting element.  

 An RACG also acts properly and cocompactly on a CAT(0) cube complex called the \textit{Davis complex}. In \cite[Proposition 2.11]{BHS}, an RACG   of linear divergence was characterized    as virtually a direct product of groups. Hence, an RACG which is not virtually a direct product of non-trivial groups contains  a rank-1 element, for its existence  is equivalent to a superlinear divergence by \cite[Theorem 2.14]{ChaSul}.
 
By the above discussion, we have the following. 
 
\begin{thm}[$\mathbb{CAT_0^1}$] \label{CATGeneric}
The set of rank-1 elements in a CAT(0) group with a rank-1 element is exponentially generic with respect to the CAT(0) metric. This includes, in particular, the following two subclass of groups:
\begin{enumerate}
\item
The action on the Salvetti complex of right-angled Artin groups that are not direct products. As a consequence, any subgroup conjugated into a join subgroup is growth-tight. 
\item
The action on the Davis complex of right-angled Coxeter groups that are not virtually a direct product of non-trivial groups.
\end{enumerate} 
\end{thm}

\subsection{{Conjugacy classes of non-contracting elements}}
We consider an implication of \ref{GenericThm} on counting conjugacy classes of non-contracting elements. To state the result, we need introduce the length of a conjugacy class, which is motivated by the length of a geodesic in a geometric setting below,  for CAT(0) spaces.

Recall that a hyperbolic  isometry on CAT(0) spaces preserves a geodesic called the \textit{axis}, on which it acts by translation (cf. \cite{BriHae}). Among the hyperbolic ones, a \textit{rank-1 isometry} requires in addition that the axis does not bound on a half-Euclidean plane.  Therefore, in a compact rank-1 manifold, each geodesic corresponds to a (conjugacy class of) hyperbolic isometry, among which we can distinguish rank-1 and non-rank-1 geodesics.  An example presented in Kneiper \cite{Kneiper1}  shows that  non-rank-1 geodesics can grow exponentially fast. However, as a corollary to his solution of a conjecture of Katok \cite[Corollary 1.2]{Kneiper1}, Kneiper showed that the number of them is stictly less than that of rank-1 geodesics. The second main result we are stating can be viewed as a generalization of this result in the setting of a SCC action with contracting elements.


Let $[g]$ denote the conjugacy class of $g\in G$. With respect to a basepoint $o$, the \textit{(algebraic)  length} of the conjugacy  class $[g]$ is defined as follows:
$$\ell^o([g])=\min \{d(o, g'o): g'\in [g]\}.$$ 
 
We remark that this does not agree with the geometric length of a conjugacy class (when viewed as a closed geodesic on manifolds, for example). Nevertheless, if the action of $G$ on $\mathrm Y$ is cocompact, then these two lengths differ up to an additive constant. With this coarse identification, our results could be applied in a geometric setting to count the number of closed geodesics. 

Similarly to that of orbital points, the \textit{growth rate of conjugacy classes} of an  infinite set $X$ in $G$ is defined as follows 
$$
\ec X:=\limsup_{n\to \infty}\frac{\log \sharp \{[g]: g\in X, \ell^o([g])\le n\} }{n}.
$$

We are ready to state the next main result, which is proved in \textsection\ref{Section3} as an initial step in the proof of \ref{GenericThm}.  

\begin{mainthm}[Conjugacy classes of non-contracting elemenets] \label{ConjugacyThm}
Assume that $G$ admits a SCC action on a geodesic metric space $(\mathrm Y, d)$ with a contracting element. Then the growth rate of   conjugacy classes of non-contracting  elements is strictly less than  $\e G$. 
 
\end{mainthm}

The first application is given to the class of groups with rank-1 elements, generalizing the above Kneiper's result on compact rank-1 manifolds to the setting of  singular spaces:

\begin{thm}[$\mathbb{CAT_0^1}$]\label{CAT01Thm} 
Suppose that $G$ acts properly and cocompactly on a $\mathrm{CAT}(0)$ space with a rank-1 element. 
Then  the growth rate of conjugacy classes of non-rank-1 elements is strictly less than $\e G$. 
\end{thm}

Kneiper's proof in the smooth case  makes essential use of conformal densities on the boundary, while our proof  replies on a growth-tightness result for a class of barrier-free set we are describing now.

 With a basepoint $o$ fixed, an element $h\in G$ is called \textit{$(\epsilon, M, g)$-barrier-free} if there exists an \textit{$(\epsilon, g)$-barrier-free} geodesic $\gamma$ with $\gamma_-\in B(o, M)$ and $\gamma_+\in B(ho, M)$:  there exists no $t\in G$ such that $d(t\cdot o, \gamma), d(t\cdot ho, \gamma)\le \epsilon$. Denote below  by $\mathcal V_{\epsilon, M, g}$   the set of $(\epsilon, M, g)$-barrier-free elements.  One of main results proven in \cite{YANG10} is that $\mathcal V_{\epsilon, M, g}$ is growth-tight, stated here in Theorem \ref{GrowthTightThm}. 

 We are now sketching the proof in the CAT(0) case,  which illustrates the basic idea of the more complex \ref{ConjugacyThm} ,
 
\begin{proof}[Sketch of the proof of Theorem \ref{CAT01Thm}]
Let $\mathcal{NC}$ denote  the set of non-rank-1 elements. We consider only hyperbolic non-rank-1 elements $g\in G$, since there are only finitely many conjugacy classes of elliptic elements \cite{BriHae}. The axis of such an element $g$ bounds a Euclidean half-plane. Since the action is cocompact, we can translate the axis of $g$ (and the bounding half-plane) to a compact neighborhood of diameter $M$ of a basepoint $o$.  To bound $\ec{\mathcal{NC}}$, it suffices to bound the cardinality of the set $X$ of non-rank-1 elements $g \in G$ with $go$ on the boundary of half-planes. 

We now fix  a rank-1 element $c$ (of sufficiently high power). Since $\langle c\rangle \cdot o$ is contracting, a segment $[o, c\cdot o]$  {is not likely} to be near any Euclidean half-plane. This implies that $g$ are $(\epsilon, M, c)$-barrier-elements and so the set $X$ is contained in $\mathcal V_{\epsilon, M, c}$.  Thus, the growth-tightness theorem \ref{GrowthTightThm} implies $\e{\mathcal V_{\epsilon, M, c}}<\e G$,  concluding the proof.
\end{proof}

In a manner parallel  to that for smooth manifolds, we could produce examples in the class of RAAGs  with exponential growth of non-rank-1 elements. An example with infinitely many ends is  the free product $\mathbb F_2 \star (\mathbb F_2\times \mathbb F_2)$, whose defining graph is the disjoint union of two vertices and a square.

To conclude this discussion, we deduce the following corollary for relatively hyperbolic groups:
\begin{thm}[$\mathbb{RelHyp}$]
In a relatively hyperbolic group $G$, the growth rate of conjugacy classes of parabolic elements and torsion elements is strictly less than $\e G$, which is computed with respect to the word metric. 
\end{thm}
By a similar trick  as above, one can construct  examples of relatively hyperbolic groups with an exponential number of classes of torsion elements. For instance, consider a free product of two groups with infinitely many classes of torsion elements. The theorem then implies that the number of torsion elements are exponentially small relative to that of hyperbolic elements.

\subsection{{Related works on genericity problem}}
Generic elements have been studied  by many authors   by undertaking a random walk on groups. Consider a probability measure $\mu$ on $G$ with finite support. Starting from the identity, one walks to subsequent elements according to the distribution of $\mu$. In $n$ steps, the distribution becomes the $n$th convolution $\mu^{*n}$.  In this regard, Maher \cite{Maher2} proved that a random element in sufficiently large subgroups in $\mathbb {Mod}$ tends to be a pseudo-Anosov element with   probability 1 as $n\to \infty$.  This result was generalized to the class of groups with ``weakly contracting'' elements by Sisto \cite{Sisto} (his definition is different from ours). 

As a matter of fact, the measure $\mu^{*n}$ is   far from the counting measure on $n$-spheres in groups. To be precise, consider the asymptotic \textit{entropy} $h(\mu)$ and \textit{drift} $\ell(\mu)$ associated with a random walk. These two quantities and the growth rate $\e G$  are related by the following \textit{fundamental inequality} (cf. \cite{Gui}):
$$
h(\mu) \le \ell(\mu) \cdot \e G
$$
Equality would  suggest  that a random walk could approach most elements. We refer to Vershik \cite{Vershik} for  related definitions and the background, and to Gou\"{e}zel et al. \cite{GMM}  for  recent progress on the strictness of this inequality in hyperbolic groups.
 
In a sense, a counting measure could reveal more information, so  the generic elements   in counting measure is usually quite different from that in a random walk. To our best knowledge,  there are very few  results in counting measures arising from word metrics.  A progress made by Caruso and Wiest \cite{CaWi} in braid groups showed that generic elements are pseudo-Anosov in the word metric with respect to  Garside's generating set. Some of the ingredients were generalized later by Wiest to treat other classes of groups in \cite{Wiest}. For instance, partial cases of Theorems \ref{RelHypGeneric} and \ref{CATGeneric} was obtained there under some automatic hypothesis.   Recently, Gehktman, Tylor and Tiozzo \cite{GTT} estbalished for word metrics the generic elements in a non-elementary hyperbolic group action. 

We emphasize that all these works assume a non-elementary  action on $\delta$-hyperbolic spaces and the existence of certain automatic structures, which are not needed in the present work. In contrast, by assuming the existence of a contracting element, our methods presented here seems to be effective in treating many genericity problems in a unified way.  Except the ones stated in \textsection \ref{SSgenericity}, we mention another result  in \cite[Proposition 2.21]{YANG10} which was proved by using similiar technics.
\begin{prop}
Assume that a finitely generated group $G$ acts properly on a geodesic metric space $(\mathrm Y, d)$ with a contracting element. Then for any finite generating set $S$,  
$$
\frac{\sharp \{g\in N(1, n): \text{g is contracting}\}}{\sharp N(1, n)}>0
$$ 
where $N(1, n)=\{g\in G: d_S(1, g)\le n\}$ where $d_S$ the corresponding word metric.
\end{prop}

This generalizes the recent result of Cumplido and Wiest \cite{CWiest} in mapping class groups that a positive proportion of elements are pseudo-Anosov. See also Cumplido \cite{Cumplido} for a similar result in Artin-Tits groups.

The structure of the rest of this paper is  as follows. The preliminary \textsection \ref{Section2} recalls necessary results proved in \cite{YANG10}. \ref{ConjugacyThm} is first proved in  section \ref{Section3}. A general theorem \ref{tightNC} is stated in \textsection \ref{Section4},  from which we deduce \ref{GenericThm}, and Theorem \ref{ModGeneric}, \;\ref{ModGeneric2}. Its proof is given by assuming an almost geodesic decomposition in Proposition \ref{almostgeodform0}, which is the goal of the following three sections \textsection\textsection \ref{Section5}, \ref{Section6}, \ref{Section7}. More preliminary is recalled in \textsection \ref{Section5} to give a brief introduction of projection complex and quasi-tree of spaces. They are used in the following \textsection \ref{Section6} to prove Lemma \ref{sameLevel}, which is the starting point of the proof of Proposition \ref{almostgeodform0} occupying the final \textsection \ref{Section7}. 

\ack
The author is grateful  to Brian Bowditch, Jason Behrstock and Ilya Gehtmann for helpful conversations.

\section{Preliminary}\label{Section2}
Most of materials are taken from the paper \cite{YANG10}, to which we refer for more details and complete proofs.  

\subsection{Notations and conventions}\label{ConvSection}
Let $(\mathrm Y, d)$ be a proper geodesic metric space. Given a point $y \in Y$ and a subset $X \subset \mathrm Y$,
let $\pi_X(y)$ be the set of points $x$ in $X$ such that $d(y, x)
= d(y, X)$. The \textit{projection} of a subset
$A \subset \mathrm Y$ to $X$ is then $\pi_X(A): = \cup_{a \in A} \pi_X(a)$.

Denote $\proj_X(Z_1, Z_2):=\diam{\pi_X({Z_1\cup Z_2})}$, which is the diameter of the projection of the union $Z_1\cup Z_2$ to $X$. So $d_X^\pi(\cdot, \cdot)$ satisfies the triangle inequality
$$
d_X^\pi(A, C) \le d_X^\pi(A, B) +d_X^\pi(B, C).
$$ 


We always consider a rectifiable path $\alpha$ in $\mathrm Y$ with arc-length parameterization.  Denote by $\len (\alpha)$ the length
of $\alpha$, and by $\alpha_-$, $\alpha_+$ the initial and terminal points of $\alpha$ respectively.   Let $x, y \in \alpha$ be two points which are given by parameterization. Then $[x,y]_\alpha$ denotes the parameterized
subpath of $\alpha$ going from $x$ to $y$. We also denote by $[x, y]$ a choice of a geodesic in $\mathrm Y$ between $x, y\in\mathrm Y$.  
\\
\paragraph{\textbf{Entry and exit points}} Given a property (P), a point $z$ on $\alpha$ is called
the \textit{entry point} satisfying (P) if $\len([\alpha_-, z]_\alpha)$ is
minimal   among the points
$z$ on $\alpha$ with the property (P). The \textit{exit point} satisfying (P) is defined similarly so that $\len([w,\alpha_+]_\alpha)$ is minimal.

A path $\alpha$ is called a \textit{$c$-quasi-geodesic} for a constant $c\ge 1$ if the following holds 
$$\len(\beta)\le c \cdot d(\beta_-, \beta_+)+c$$
for any rectifiable subpath $\beta$ of $\alpha$.

Let $\alpha, \beta$ be two paths in $Y$. Denote by $\alpha\cdot \beta$ (or simply $\alpha\beta$) the concatenated path provided that $\alpha_+ =
\beta_-$.

Let $f, g$ be real-valued functions with domain understood in
the context. Then $f \prec_{c_i} g$ means that
there is a constant $C >0$ depending on parameters $c_i$ such that
$f < Cg$.  The symbols    $\succ_{c_i}  $ and $\asymp_{c_i}$ are defined analogously. For simplicity, we shall
omit $c_i$ if they are   universal constants.
\subsection{Contracting elements}
\begin{defn}[Contracting subset]\label{ContrDefn}
For given $C\ge 1$, a subset $X$ in $\mathrm Y$ is called $C$-\textit{contracting} if for any geodesic $\gamma $ with $d(\gamma,
X) \ge C$, we have
$$\proj_{X} (\gamma)  \le C.$$
 A
collection of $C$-contracting subsets is referred to
as a $C$-\textit{contracting system}.
\end{defn}




We collect a few properties that will be used often later on.
\begin{prop}\label{Contractions}
Let $X$ be a contracting set.
\begin{enumerate}
\item
(Quasi-convexity) \label{qconvexity}  $X$ is \textit{$\sigma$-quasi-convex} for a function $\sigma: \mathbb R_+ \to \mathbb R_+$: given $c \ge
1$,  any $c$-quasi-geodesic with endpoints in $X$ lies in the
neighborhood $N_{\sigma(c)}(X)$. 
\item
(Finite neighborhood) \label{nbhd}  Let $Z$ be a set with finite Hausdorff distance to $X$. Then $Z$ is contracting.

\item
(Subpaths) \label{subpath} If $X$ is a quasi-geodesic, then any subpath of $X$ is contracting with contraction constant depending only on $X$. 

There exists $C>0$ such that the following hold:

\item
(1-Lipschitz projection) \label{1Lipschitz} $\proj_X(\{y, z\})\le d(y, z)+C$.

\end{enumerate}
\end{prop}
\begin{proof}
Except Assertion (3),  the others are straightforward applications of contracting property. The (3) for geodesics in CAT(0) spaces can be found in \cite[Lemma 3.2]{BF2}; here we  provide a proof in this general setting.  

Assume that $\gamma:=X$ is a $C$-contracting $c$-quasi-geodesic   for some $c, C>0$.  We first observe the following.
\begin{claim}
There exists $D=D(c, C)>0$ such that any subpath $\alpha$ of $\gamma$ has at most a Hausdorff distance $D$ to a geodesic $\tilde \alpha$ with $d(\alpha_-, \tilde \alpha_-), d(\alpha_-, \tilde\alpha_-)\le 2C$.
\end{claim}
\begin{proof}[Proof of the claim]
Indeed,  by the quasi-convexity (\ref{qconvexity}), there exists $\sigma=\sigma(C)>0$ such that $\tilde \alpha\subset N_\sigma(\gamma)$. We shall only prove $\tilde \alpha \subset N_D(\alpha)$; the other inclusion follows from this one by a standard argument using the connectedness of $\tilde \alpha$. 

Without loss of generality, we prove $\tilde \alpha \subset   N_{\sigma}(\alpha)$ by assuming that $\len(\alpha)>c(2\sigma+1).$ Argue by contradiction. If $\tilde \alpha \subsetneq   N_{\sigma}(\alpha)$, there exists a non-empty open (connected) interval $I$ in $\tilde \alpha$ such that $I \cap N_{\sigma}(\alpha)=\emptyset$. 
By the fact $\tilde \alpha\subset N_\sigma(\gamma)$, we must have $I \subset N_{\sigma}([\gamma_-, \alpha_+]_\gamma) \cup N_{\sigma}([\alpha_-, \gamma_+]_\gamma).$ The connectivity of $I$ then implies that $I \cap N_{\sigma}([\gamma_-, \alpha_-]_\gamma) \cap N_{\sigma}([\alpha_+, \gamma_+]_\gamma)\ne \emptyset$: there exists $x\in I$ such that $d(x, [\gamma_-, \alpha_-]_\gamma), d(x, [\alpha_+, \gamma_+]_\gamma) \le \sigma$. As a consequence, $\alpha$ is contained in a $c$-quasi-geodesic with two endpoints within a $2\sigma$-distance, so gives that $\len(\alpha)\le c(2\sigma+1).$ This is a contradiction with $\len(\alpha)>c(2\sigma+1)$, so $\tilde \alpha \subset   N_{\sigma}(\alpha)$ is proved.   The proof of the claim is thus finished.
\end{proof}  

We now prove that $\alpha$ is $2(D+C)$-contracting. Consider a geodesic $\beta$ such that $\beta\cap N_D(\alpha)=\emptyset$. Let $x, y \in \pi_\alpha(\beta)$ be  projection points of $\tilde x, \tilde y \in \beta$ respectively such that $d(x, y)=\proj_\alpha(\beta)$. The goal is to show that $d(x, y)\le 2(D+C)$.
   
Without loss of generality, assume that the projection $\pi_\gamma(\beta)$ is not entirely contained in $\alpha$; otherwise, the contracting property of $\alpha$ would follow from the one of $\gamma$.  For definiteness,  assume that there exists a point $z\in \pi_\gamma(\beta)$ lies  \textit{on the left side} of $\alpha_-$: $z \in [\gamma_-, \alpha_-)_\gamma$; and by switching $x, y$ if necessary, assume further that $x$ is on the left of $y$. 

Let $w \in \pi_\gamma(\tilde y)$ be a  project point of $\tilde y$ to $\gamma$. The contracting property of $\gamma$ implies  $d(w, [\tilde y, y])\le 2C$. Noticing that $w, y\in \alpha$, we deduce from  the \textbf{Claim} above  that $[w, y]_\gamma$ is contained in a $D$-neighborhood of $[y, \tilde y]$.

We observe that $w\in N_C([\gamma_-, \alpha_-]_\gamma)$. Indeed, to derive a contradiction, assume that $w$ lies on $[\alpha_-, \gamma_+]_\gamma$ and $d(w, \alpha_-)>C$. Noting that $z$ lies on $[\gamma_-, \alpha_-)_\gamma$ and $z\in \pi_\gamma(\beta)$, we have $d(z, w)\ge d(w, \alpha_-)>C$. By the contracting property, the points $z, w\in \pi_\gamma(\beta)$ with $d(z, w)>C$  implies $N_C(\gamma)\cap \beta\ne \emptyset$. Moreover,   a projection argument shows that $d(z, \beta), d(w, \beta)\le 2C$. So by the \textbf{Claim} above, $[z, w]_\gamma$ lies within the $D$-neighborhood of $\beta$. Since $\alpha_-\in [z, w]_\gamma$, we get $N_D(\alpha)\cap \beta \ne \emptyset$, a contradiction with the hypothesis. 

It is therefore proved that  either $w$ is on the left of $\alpha_-$, or $d(w, \alpha_-)\le C$. Recall that $x \in [\alpha_-, y]_\alpha$, and $[w, y]_\gamma \subset N_D([y, \tilde y])$ as proved above. Examining the aformentioned relative position   of $w$ to  $\alpha_-$, we then obtain that in the former case, $d(x, [\tilde y, y])\le D$ and in the later case, $d(x, [\tilde y, y])\le D+C$.  In both cases, let $x'\in [\tilde y, y]$ so that $d(x, x')\le D+C$. 

To conclude the proof, we verify that $d(x, y)\le 2(D+C)$: since $y$ is a shortest point on $\alpha$ to $\tilde y$, we have $d(x, x')\ge d(y, x')$. Hence, $d(x, y)\le d(x, x')+d(x', y) \le 2(D+C),$ proving  that $\alpha$ is $2(D+C)$-contracting. 
\end{proof}

A  contracting system has a \textit{$\mathcal R$-bounded intersection} property for a function $\mathcal R: \mathbb R_{\ge 0}\to \mathbb R_{\ge 0}$ if the
following holds
$$\forall X\ne X'\in \mathbb X: \;\diam{N_r (X) \cap N_r (X')} \le \mathcal R(r)$$
for any $r \geq 0$. This property is, in fact, equivalent to a \textit{bounded intersection
property} of $\mathbb X$:  there exists a constant $B>0$ such that the
following holds
$$\proj_{X'}(X) \le B$$
for $X\ne X' \in \mathbb X$. See \cite{YANG6} for further discussions.

Recall that $G$ acts properly by isometry on a geodesic metric space $(\mathrm Y, d)$.   An  element $h \in G$ is called  
\textit{contracting} if the orbit $\langle h \rangle\cdot o$ is contracting, and the orbital map
\begin{equation}\label{QIEmbed}
n\in \mathbb Z\to h^no \in \mathrm Y
\end{equation}
is a quasi-isometric embedding.  The set of contracting elements is preserved under conjugacy.

Given a contracting element $h$, there exists a maximal elementary group $E(h)$ containing $\langle h \rangle$ as a finite index subgroup. Precisely, 
$$
E(h)=\{g\in G: \exists n > 0,\; (gh^ng^{-1}=h^n)\; \lor\;  (gh^ng^{-1}=h^{-n})\}.
$$

In what follows, the contracting subset 
\begin{equation}\label{axisdefn}
\ax(h)=\{f \cdot o: f\in E(h)\}
\end{equation} shall be called the \textit{axis} of $h$.  Hence, the collection $\{g \ax(h): g\in G\}$ is a contracting system with bounded intersection.  

Two  contracting elements $h_1, h_2\in G$  are called \textit{independent} if the collection $\{g\ax(h_i): g\in G;\ i=1, 2\}$ is a contracting system with bounded intersection.

\subsection{Admissible paths}
The notion of an admissible path is defined relative to   a  contracting system $\mathbb X$ in $\mathrm Y$. Roughly speaking, an admissible path can be thought of as a
concatenation of quasi-geodesics which travels alternatively near
contracting subsets and leave them in an orthogonal way. 


\begin{defn}[Admissible Path]\label{AdmDef}
Given $D, \tau \ge 0$ and a function $\mathcal R: \mathbb R_{\ge 0} \to \mathbb R_{\ge 0}$, a path $\gamma$  is called \textit{$(D,
\tau)$-admissible} in  $\mathrm Y$, if  the path $\gamma$ contains a sequence of disjoint geodesic subpaths $p_i$ $(0\le i\le n)$ in this order, each associated to a contracting subset $X_i \in \mathbb X$, with the following   called \textit{Long Local} and \textit{Bounded Projection} properties:
\begin{enumerate}

\item[\textbf{(LL1)}]
Each $p_i$ has length bigger than  $D$, except that  $(p_i)_- =\gamma_-$ or $(p_i)_+=\gamma_+$,

\item[\textbf{(BP)}]
For each $X_i$,  we have 
$$
\proj_{X_i}((p_{i})_+,(p_{i+1})_-)\le \tau
$$
and 
$$
\proj_{X_i}((p_{i-1})_+, (p_{i})_-)\le \tau
$$
when $(p_{-1})_+:=\gamma_-$ and $(p_{n+1})_-:=\gamma_+$ by convention.

\item[\textbf{(LL2)}]
Either $X_i, X_{i+1}$ has $\mathcal R$-bounded intersection or $d((p_i)_+, (p_{i+1})_-)>D$,


\end{enumerate}
\paragraph{\textbf{Saturation}} The collection of $X_i \in \mathbb X$ indexed as above, denoted by $\mathbb X(\gamma)$, will be referred to as contracting subsets for $\gamma$. The union of all $X_i \in \mathbb X(\gamma)$ is called the \textit{saturation} of $\gamma$. 
\end{defn}

The set of endpoints of $p_i$ shall be refered to as the \textit{vertex set} of $\gamma$. We call $(p_{i})_-$ and $(p_{i})_+$ the corresponding \textit{entry vertex} and \textit{exit vertex}   of $\gamma$ in $X_i$. (compare with entry and exit points in subSection \ref{ConvSection})


By definition, a sequence of points $x_i$ in a path $\alpha$   is called \textit{linearly ordered} if $x_{i+1}\in [x_i, \alpha_+]_\alpha$ for each $i$. 

\begin{defn}[Fellow travel]\label{Fellow}
Assume that $\gamma = p_0 q_1 p_1 \cdots q_n p_n$ is a $(D, \tau)$-admissible
path, where each $p_i$ has two endpoints in $X_i \in \mathbb X$. The paths $p_0, p_n$ could be trivial. 

Let
$\alpha$ be a path such that $\alpha_- = \gamma_-,
\alpha_+=\gamma_+$. Given $\epsilon >0$, the path $\alpha$ \textit{$\epsilon$-fellow travels} $\gamma$ if there exists a sequence of linearly ordered points $z_i,
w_i$ ($0 \le i \le n$) on $\alpha$ such that  
$d(z_i, (p_{i})_-) \le \epsilon, \;d(w_i, (p_{i})_+) \le \epsilon.$
\end{defn}

The basic fact  is that a ``long" admissible path is a quasi-geodesic.
 
\begin{prop}\label{admissible} 
Let $C$ be the contraction constant of $\mathbb X$. For any $\tau>0$, there are constants $B=B(\tau), D=D(\tau), \epsilon = \epsilon(\tau), c = c(\tau)>0$ such that the following
holds.

Let $\gamma$ be a $(D, \tau)$-admissible path and $\alpha$ a geodesic  between $\gamma_-$ and
$\gamma_+$. Then
\begin{enumerate}
\item
For    a contracting subset $X_i \in \mathbb X(\gamma)$ with $0 \le i \le
n$,
$$\proj_{X_i}(\beta_1) \le B,\; \proj_{X_i}(\beta_2)  \le B$$ where $\beta_1 =[\gamma_-, (p_i)_-]_\gamma, \beta_2 =[(p_i)_+, \gamma_+]_\gamma$. 
\item
$\alpha \cap N_C(X) \ne \emptyset$ for every $X\in \mathbb X(\gamma)$.
\item
$\alpha$  $\epsilon$-fellow travels $\gamma$. In particular,  $\gamma$ is a $c$-quasi-geodesic.
\end{enumerate}
\end{prop}

 The main use of this lemma (the second statement) is to construct the following type of paths in verifying that an element is contracting.

\begin{defn}\label{uniformadmissible}
Let $L, \Delta>0$. With notations in definition of a $(D, \tau)$-admissible path $\gamma$, if the following holds $$|d((p_{i+1})_-, (p_{i})_+)-L|\le \Delta$$ for each $i$, we say that $\gamma$ is a \textit{$(D, \tau, L, \Delta)$-admissible} path.  
\end{defn}

 \begin{prop}\label{saturation}
Assume that $\mathbb X$ has bounded intersection in   admissible paths considered in the following statements.
For any $\tau>0$ there exists $D=D(\tau)>0$ with the following properties.
\begin{enumerate}
\item
For any $L,\Delta>0$, there exists $C=C(L, \Delta)>0$ such that the saturation of a $(D, \tau, L, \Delta)$-admissible path is  $C$-contracting.  
\item
For any $L,\Delta, K>0$, there exists $C=C(L, \Delta, K)>0$  such that if the entry and exit vertices of a $(D, \tau, L, \Delta)$-admissible path $\gamma$ in each $X\in \mathbb X(\gamma)$ has distance bounded above by $K$, then $\gamma$ is $C$-contracting.
\end{enumerate}
\end{prop}

\subsection{Statistically convex-cocompact actions and growth-tightness theorem}\label{SSSCC}
In this subsection, we recall the definition of statistically convex-cocompact actions, which is understood as a statistical version of convex-cocompact actions.     By abuse of language, a geodesic between two sets $A$ and $B$ is a geodesic $[a, b]$ between $a\in A$ and $b\in B$.

 Given constants $0\le M_1\le M_2$, the \textit{a concave region}   $\mathcal O_{M_1, M_2}$ consists of the set  of elements $g\in G$ such that there exists some geodesic $\gamma$ between $B(o, M_2)$ and $B(go, M_2)$ with the property that the interior of $\gamma$ lies outside $N_{M_1}(Go)$.

\begin{defn}[statistically convex-cocompact action]\label{StatConvex}
If there exist two positive constants  $M_1, M_2>0$ such that $\e {\mathcal O_{M_1, M_2}} < \e G$, then the action of $G$ on $\mathrm Y$ is called \textit{statistically convex-cocompact (SCC)}.  

\end{defn} 
 
We are interested in the following important examples of SCC actions:
\begin{examples}
\begin{enumerate}
\item
Any proper and cocompact group action on a geodesic metric space. 
 
\item
The action of relatively hyperbolic groups with parabolic gap property on a hyperbolic space (cf. \cite{DOP}).
\item
The action of mapping class groups on Teichm\"{u}ller spaces is SCC (cf. \cite{ACTao}).  
\end{enumerate}
\end{examples}

In applications, since $\mathcal O_{M_2, M_2} \subset \mathcal O_{M_1, M_2}$, we can assume that $M_1=M_2$ and henceforth, denote $\mathcal O_M:=\mathcal O_{M, M}$ for easy notations.  

We remark that the definition of a SCC action is independent of the choice of basepoint $o$, when there exists a contracting element. Namely, for any basepoint $o$, there exist $M_1, M_2>0$ such that $\e {\mathcal O_{M_1, M_2}} < \e G$. See Lemma 6.1 in \cite{YANG10}.

By definition, the union of two growth-tight (resp.  growth-negligible) sets is  growth-tight (resp. growth-negligible).  
The main result of this section shall provide a class of  growth-tight sets. These growth-tight sets are closely related to a notion of a barrier we are going to introduce now.

\begin{defn}\label{barriers}
Fix constants $\epsilon, M>0$ and a set $P$ in $G$. \begin{enumerate}
\item(Barrier/Barrier-free geodesic)
Given $\epsilon>0$ and $f\in P$, we say that a geodesic $\gamma$ contains an \textit{$(\epsilon, f)$-barrier}   if there exists    a element $h \in G$ so that 
\begin{equation}\label{barrierEQ}
\max\{d(h\cdot o, \gamma), \; d(h\cdot fo, \gamma)\}\le \epsilon.
\end{equation}
If 
 no such $h \in G$ exists so that (\ref{barrierEQ}) holds, then $\gamma$ is called \textit{$(\epsilon, f)$-barrier-free}.  

 Generally, we say $\gamma$ is \textit{$(\epsilon, P)$-barrier-free} if it is $(\epsilon,f)$-barrier-free for some $f\in P$. An obvious fact is that  any subsegment of $\gamma$ is also $(\epsilon, P)$-barrierr-free.
\item(Barrier-free element)
An element $g\in G$ is \textit{$(\epsilon, M, P)$-barrier-free} if there exists  an $(\epsilon, P)$-barrier-free geodesic between $B(o, M)$ and $B(go, M)$. 
 
\end{enumerate}
\end{defn}
 
The following result is proved in \cite[Theorem C]{YANG10}. We remark that the constant  $M$ can be chosen as big as necessary, which is guaranteed by Lemma 6.1 in \cite{YANG10}. 
\begin{thm}[Growth tightness] \label{GrowthTightThm}
Suppose that $G$ has an SCC action on a geodesic space $(\mathrm Y, d)$ with a contracting element. Then there exist  constants $\epsilon, M>0$ such that for any given $g \in G$, we have
$\mathcal V_{\epsilon, M, g}$ is growth-tight. If the action is proper, then $\mathcal V_{\epsilon, M, g}$ is growth-negligible.   
  
\end{thm}

We sketch the proof at the convenience of the reader, and refer to \cite[Section 4]{YANG10} for complete details.
 
\begin{proof}[Sketch of the proof] 
 Let $\mathcal B$ be a \textit{maximal $R$-separated} subset in $\mathcal A:=\mathcal V_{\epsilon, M, g}$ so that 
\begin{itemize}
\item
for any distinct $a, a'
\in \mathcal B$,  
$d(a o,   a' o)
> R$, and
\item
for any $x \in \mathcal V_{\epsilon, M, g}$, there
exists $a \in \mathcal B$ such that $d(xo,  ao) \le R$. 
\end{itemize}

Denote by $\mathbb W(\mathcal A)$ the set of all (finite) words over $\mathcal A$.    We defined an \textit{extension map} $\Phi: \mathbb W(A) \to G$ as follows: given a word $W=a_1a_2\cdots a_n \in \mathbb W(A)$, set $$\Phi(W)=a_1 \cdot f_1 g f_1' \cdot a_2 \cdot f_2 g f_2'\cdot \cdots \cdot a_{n-1}\cdot f_{n-1} g f_{n-1}' \cdot a_n \in G,$$ where $f_i, f_i'\in F$ are supplied by the extension lemma in \cite{YANG10} such that $\Phi(W)$ labels a $(D, \tau)$-admissible path. Consider  $X:=\Phi(\mathbb W(\mathcal B))$. The key fact is that $\Phi: \mathbb W(\mathcal B) \to G$ is injective.

Consider the  Poincar\'e series 
$$\PS{\Gamma} = \sum\limits_{g \in \Gamma} \exp(-sd(o, go)), \; s \ge 0,$$
which  diverges for $s<\e \Gamma$ and converges for $s>\e \Gamma$. 
  
Note that $\PS{\mathcal B} \asymp \PS{\mathcal A}$, whenever they are finite, and so $$\e {\mathcal B} = \e {\mathcal A}.$$ 

If the action is assumed to only be proper, then with the critical gap criterion in \cite[Lemma 2.23]{YANG10}, the   injectivity assertion of $\Phi$ implies that $\PS{\mathcal B}$ converges at $s=\e G$: $\mathcal B$ is a growth-negligible set.

If the action is assumed to  be SCC, we were able to prove that $\PS{\mathcal A}$ and thus $\PS{\mathcal B}$
are divergent at $s = \e {\mathcal A}$. Again, by critical gap criterion in \cite{YANG10}, we proved that $\e {X} > \e {\mathcal B}$ and so $\e{G} \ge
\e {X} > \e {\mathcal A}$: $\mathcal A$ is a growth-tight set. 
\end{proof}

\section{Conjugacy classes of non-contracting elements}\label{Section3}
Recall that the \textit{growth rate of conjugacy classes} of  an  infinite set $X$ is defined as 
$$
\ec X=\limsup_{n\to \infty}\frac{\log \sharp \{[g]: g\in X, \ell^o([g])\le n\} }{n}
$$
where $[g]$ denotes the conjugacy class of $g$ in $G$. An element $h$  in $[g]$  is called \textit{minimal} if $d(o, ho)=\ell^o([g])$. 
Denote by $\mathcal {NC}$ the set  of non-contracting elements in $G$ for the action of $G$ on a geodesic metric space $(\mathrm Y, d)$. 

The goal of this section is the following result.

\begin{thm}\label{NCisBarrierfree}
Assume that $G$ admits a proper action on $\mathrm Y$. Fix a constant $\epsilon>0$ and a contracting element $f\in G$. For any sufficiently large $M>0$, there exists an integer $n>0$ such that each element $g\in \mathcal NC$ admits a minimal conjugacy representative in $\mathcal V_{\epsilon, M, f^n}.$ 
\end{thm}

Assuming SCC action,  \ref{ConjugacyThm} follows immediately from it.
 
\begin{proof} [Proof of Theorem \ref{ConjugacyThm}]
 Let $\epsilon, M>0$ be the constants given by the growth-tightness Theorem \ref{GrowthTightThm}, and the constant $M$  is also assumed to satisfy Theorem \ref{NCisBarrierfree}. It follows immediately from Theorems \ref{NCisBarrierfree} and \ref{GrowthTightThm} that $\ec {\mathcal {NC}}<\e G$. 
\end{proof}

The rest of this section is devoted to the proof of Theorem \ref{NCisBarrierfree}. In the next three lemmas, we would like to first identify two subsets of $\mathcal NC$, and prove that they belong to $\mathcal V_{\epsilon, M, f^n}.$ In the last ingredient, we show in Proposition \ref{NCRep} that, up to finitely many exceptions, these two subsets comprises the entire $\mathcal NC$.
 


For any $M, D>0$, let $\mathcal K_{M, D}$ be the set of elements $h$ in $G$ such that any subpath of length $D$ in a geodesic $[o, ho]$ is not contained in $N_{M}(Go)$. Noting the similarity between $\mathcal K_{M, 0}$ and $\mathcal O_{M}$, the following fact is not surprising.

\begin{lem}\label{OArea}
For any sufficiently large $M\gg 0$ and any $D>0$, there exists $n>0$ such that $\mathcal K_{M, D}\subset \mathcal V_{\epsilon, M, f^n}$.
\end{lem}
\begin{proof}
Since the axis $\ax(f)$ is contracting, it follows by Proposition \ref{Contractions}.\ref{qconvexity} that $\ax(f)$ is $\sigma$-quasi-convex for a function $\sigma: \mathbb R_{>0}\to \mathbb R_{>0}$.   
Suppose, to the contrary, that some $h\in \mathcal K_{M, D}$ contains an $(\epsilon, f^n)$-barrier $g\in G$, so $d(go, [o, ho]), \;d(gf^no, [o, ho])\le \epsilon$. As a consequence, there exists a subsegment $\alpha$ of $[o, ho]$ with their endpoints $\alpha_\pm \in N_\epsilon(\ax(f))$ such that  $$\len(\alpha)>d(o, f^no)-2\epsilon.$$
From the $\sigma$-quasi-convexity of  $\ax(f),$ we obtain 
  $$\alpha \subset N_\sigma(\ax(f))\subset N_\sigma(Go).$$ 

We choose now $M>\sigma$, and $n$ large enough such that $d(o, f^no)>D+2\epsilon$. Thus, the subpath $\alpha$ with length $\len(\alpha)>D$ is contained inside $N_M(Go)$. This gives a contradiction to the definition of $h\in \mathcal K_{M, D}$. Hence,  $\mathcal K_{M, D}\subset \mathcal V_{\epsilon, M, f^n}$.
\end{proof}

Given $D, C>0$, a geodesic $\gamma$ is called \textit{$D$-local $C$-non-contracting} if any connected subsegment  of $\gamma$ with length $D$ contained in $N_M(Go)$ is not $C$-contracting.   Denote by $\mathcal P_{D, C}$ the set of  $h \in G$ such that $[o, ho]$ is $D$-local $C$-non-contracting. 

The following lemma makes essential use of the hypothesis (\ref{QIEmbed}). However, it is noteworthy that most results in this paper do not require it. 
\begin{lem}\label{contractingsegment}
There exists $C>0$ depending on $\epsilon$ and $f$ such that any geodesic between $B(o,\epsilon)$ and $B(f^no, \epsilon)$ is $C$-contracting for $n\gg 0$.
\end{lem}
\begin{proof}
By the hypothesis (\ref{QIEmbed}), we have that  $n\in \mathbb Z \to f^no \in \mathrm Y$ is a quasi-isometric embedding with a contracting image. That is to say,  the path $\gamma$ labeled by $\{f^n: n\in \mathbb Z\}$ is     a contracting quasi-geodesic. By  Assertion  (\ref{subpath})  of Proposition \ref{Contractions}, a subpath of a contracting quasi-geodesic is uniformly contracting. So for any geodesic $\alpha$ between $B(o,\epsilon)$ and $B(f^no, \epsilon)$, it has finite Hausdorff distance to a subpath of  $\gamma$ which is contracting.  Because the contracting property is preserved up to a finite Hausdorff distance by (\ref{nbhd}) of Proposition \ref{Contractions}, we conclude that  $\alpha$ is contracting as well.  
\end{proof}

\begin{lem}\label{DCArea}
For any $D>0$, there exists $n$ such that $\mathcal P_{D,C} \subset \mathcal V_{\epsilon, M, f^n}$, where $C>0$  is given by Lemma \ref{contractingsegment}.
\end{lem}
\begin{proof}
The proof is similar to that of Lemma \ref{OArea}. We give it for completeness.

Argue by way of contradiction. For any $h\in \mathcal P_{D,C}\setminus \mathcal V_{\epsilon, M, f^n}$, a geodesic  between $B(o, M)$ and $B(ho, M)$ contains an $(\epsilon, f^n)$-barrier. In particular, let us consider the geodesic $\gamma=[o, ho]$, so by definition \ref{barriers}, there exists a subsegment $\alpha$ of $\gamma$ such that $\len(\alpha)>d(o, f^no)-2\epsilon$. In addition, the two endpoints of $\alpha$ lies in a $\epsilon$-neighborhood of $\ax(f)$ so by Lemma \ref{contractingsegment}, we see that  $\alpha$ is $C$-contracting. However, by choosing $n\gg 0$ large enough such that $$\len(\alpha)>d(o, f^no)-2\epsilon>D,$$ we got a contradiction, because $[o, ho]$  is $D$-local $C$-non-contracting for $h\in \mathcal P_{D,C}$. The proof is then complete.
\end{proof}

This is the main technical part in proving Theorem \ref{NCisBarrierfree}. The idea behind the proof is the well-known fact that the minimal conjugacy representative   generates a local geodesic path.   

\begin{prop}\label{NCRep}
For any $C>0$, there exists $D=D(C)>0$ such that  every $g \in \mathcal {NC}$ is conjugated to an element $h \in G$ which has one of the following properties
\begin{enumerate}
\item
$h \in \mathcal P_{D, C}$,   
\item
$d(o, ho)\le 4D$,
\item
$h\in \mathcal K_{M, D}$.
\end{enumerate}
\end{prop}
\begin{proof} 
Let $h$ be a \textit{minimal} element in $[g]$ such that $d(o, ho)=\ell^o([g])$. Since contracting elements are preserved under conjugacy, the element $h$ is thus  non-contracting. We fix a constant $D>2M+2C$ in the proof.
Assuming that  $h\notin \mathcal K_{M, D}$ and $d(o, ho)>4D$, we shall prove that $h \in \mathcal P_{D, C}$. 

Suppose, to the contrary, that there exists a subsegment $\alpha$ of length $D$ in $[o, ho]$ such that $\alpha\subset N_M(Go)$ and $\alpha$ is $C$-contracting. Consider the path defined by $$\gamma:=\bigcup_{n\in \mathbb Z} h^n(\alpha\cdot [\alpha_+, h\alpha_-]).$$  The idea of proof is to show that $\gamma$ is $(D, \tau)$-admissible with a contracting system $\mathbb X(\gamma):=\{h^{n}\alpha: n\in \mathbb Z\}$, which will imply   $h$ to be contracting so gives a contradiction to $h\in \mathcal NC$. We now verify the following conditions appearing in Definition \ref{AdmDef}.

\underline{Condition (\textbf{LL1})}: it holds by construction: $\len(\alpha)= D$.

\underline{Condition (\textbf{LL2})}: up to a translation by some power of $h$, it suffices to prove $d(\alpha_+, h\alpha_-)>D$. For that end, we argue by contradiction by assuming the second inequality in the following
\begin{equation}\label{absdiffEQ}
|d(ho, \alpha_+)-d(ho, h\alpha_-)|\le d(\alpha_+, h\alpha_-)\le D.
\end{equation}

Let $m$ be the middle point of $[o, ho]$ so that $d(o, m)=d(m, ho)$. Thus, noting $\alpha_\pm \in [o, ho]$, we deduce from (\ref{absdiffEQ}): 
 $$
 \begin{array}{llr}
 D=\ell(\alpha)&=d(o, ho)-d(o, \alpha_-)-d(\alpha_+, ho)\\
 &\ge d(o, ho)-2d(o, \alpha_-)-D.   
 \end{array}
 $$ which yields  $d(o, \alpha_-)+\len(\alpha) =d(o, \alpha_-)+D \ge d(o, m)$ so $\alpha_+\in [m, ho]$. 
 
Moreover, we see $\alpha_- \in [o, m]$. Otherwise, if $\alpha_- \in (m, ho]$, then $d(ho, \alpha_+)+\len(\alpha)=d(ho, \alpha_+)+D<d(o, m)<d(o, \alpha_-)$. This is a contradiction to  (\ref{absdiffEQ}).   As a result, we have $m\in \alpha$ and thus, $$\max\{d(\alpha_-, m),\; d(\alpha_+, m)\}\le D$$
 from which we infer  
\begin{equation}\label{dmhmEQ}   
 d(m, hm)\le d(m, \alpha_-)+d(\alpha_-, h\alpha_+)+d(h\alpha_+, hm)\le 3D,
\end{equation}
 for $d(\alpha_+, h\alpha_-)\le D$ was assumed in  (\ref{absdiffEQ}).   

Now, for $\alpha\subset N_M(Go)$, let $k\in G$ such that $d(ko,m)\le M$. By (\ref{dmhmEQ}), $d(ko, hko)\le 2M+3D$. However, $d(o, ho)\le d(ko, hko)$ holds by the minimal choice of $h$. By the choice of $D>2M$, we got a contradiction with $d(o, ho)>4D$. So the Condition (\textbf{LL2}) is fulfilled.

\begin{figure}[htb] 
\centering \scalebox{0.8}{
\includegraphics{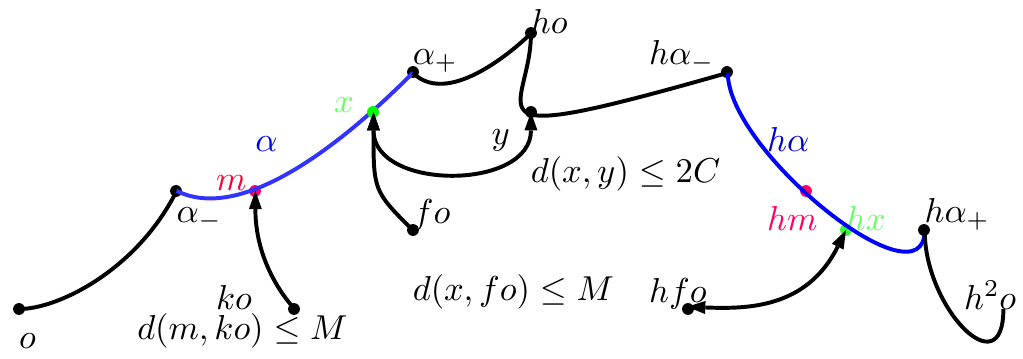} 
} \caption{Proof of Proposition \ref{NCRep}} \label{fig:fig2}
\end{figure}
 
\underline{Condition (\textbf{BP})}: Up to a translation, it suffices to  prove $$\proj_{\alpha}(\{\alpha_+, h\alpha_-\})\le \tau$$ and $$\proj_{\alpha}(\{h^{-1}\alpha_+, \alpha_-\})\le \tau$$ for some constant $\tau>0$ given below. We only prove  the first inequality; the second one is symmetric. 

To derive a upper bound of $\proj_{\alpha}(\{\alpha_+, h\alpha_-\})$, we will do it for $\proj_{\alpha}(\{\alpha_+, ho\})$ and $\proj_\alpha(\{ho, h\alpha_-\})$.  
First of all, since $\alpha$ is a contracting subsegment of the geodesic $[o, ho]$, we see $\proj_{\alpha}(\{\alpha_+, ho\})\le 2C$ by a projection argument: project $ho$ to a point $z\in \alpha$ then by $C$-contracting property of $\alpha$, we have $d(z, [\alpha_+, ho])\le 2C$ so $d(z, \alpha_+)\le 2C$.

Note that $[ho, h\alpha_-]\cap N_C(\alpha)=\emptyset$ implies $\proj_\alpha(\{ho, h\alpha_-\})\le C$ by the contracting property. Henceforth, it remains to consider the case  $[ho, h\alpha_-]\cap N_C(\alpha)\ne \emptyset$. 

We prove  now  that $\alpha_+\in  [m, ho]$. Indeed,  otherwise we have $\alpha_+ \in [o, m)$, then $d(o, m)<d(\alpha_+, ho)$. It then follows: 
$$d(ho, h\alpha_-)=d(o, \alpha_-)=d(o, m)-\len(\alpha)-d(\alpha_+,m)< d(ho, \alpha_+)-D.$$ On the other hand, since $[ho, h\alpha_-]\cap N_C(\alpha)\ne \emptyset$ is assumed, we see $d(\alpha_+, ho)\le C+d(ho, h\alpha_-)$. By the choice of the constant $D>C$, this contradicts to the above inequality. Hence, $\alpha_+\in [m, ho]$ is proved. 

Before proceeding further, we note the following two additional facts.
\begin{enumerate}
\item
We project $h\alpha_-$ to a point $x\in \alpha$.  By the assumption that $\alpha$ is $C$-contracting and $[ho, h\alpha_-]\cap N_C(\alpha)\ne \emptyset$, we see that there exists $y\in [ho, h\alpha_-]$ such that $d(x, y)\le 2C$.
\item
Taking $\alpha\subset N_M(Go)$ into account, there exist $M>0$ and $f\in G$ such that $d(x, fo)\le M$.  
Thus, 
\begin{equation}\label{ygxEQ}
\begin{array}{ll}
d(fo, hfo)&\le d(fo, x)+d(x, y)+d(y, hx)+d(hx, hfo)\\
&\le 2M+2C+d(y, hx).
\end{array}
\end{equation}
\end{enumerate}

By the observation above that $\alpha_+\in [m, ho]$, we examine the following two cases, the second case of which will  be proven impossible.

\textbf{Case 1.} Assume that $\alpha_-\in [m, ho]$.  By the position of $y\in [ho, h\alpha_-]$ and $hx\in [h\alpha_-, h^2o]$, we see 
\begin{equation}\label{yhxEQ}
\begin{array}{rlr}
d(y, hx)=&d(ho, h^2o)-d(hx, h^2o)-d(y, ho)& \\
\le& d(o, ho)-2d(y, ho)+2C,
\end{array}
\end{equation}
where we used $d(x,y)\le 2C$.
By the minimal choice of $h$, we have $d(o, ho)\le d(fo, hfo)$ so by combining (\ref{ygxEQ}) and (\ref{yhxEQ}),  $d(y, ho)\le M+2C$. Since $d(x,y)\le 2C$ and $\alpha_+\in [x, ho]$, we have $d(x, \alpha_+)\le d(x, ho)\le M+4C$. Because $\alpha$ is a subsegment of a geodesic $[o, ho]$, the endpoint $\alpha_+$ of $\alpha$ must be a projection point of $ho$.   Hence, $\proj_{\alpha}(\{ho, h\alpha_-\})\le d(\alpha_+, x) \le M+4C$.

\textbf{Case 2.} Otherwise, we have $m\in \alpha$ and then $d(x, m)\le D$ for $x\in \alpha$. Since $d(x,y)\le 2C$ and $d(o,m)=d(ho, m)$, we have 
$$
\begin{array}{rl}
|d(ho, y)-d(ho,hm)|=&|d(ho, y)-d(ho,m)|\\
 \le &d(y, m)\le 2C+D
\end{array}
$$
yielding 
$$
\begin{array}{rl}
d(y, hx)=|d(ho, hx)-d(ho, y)| \le &|d(ho, y)-d(ho,hm)|+d(x,m) \\
\le & 2D+2C.
\end{array}
$$ We thus derive from (\ref{ygxEQ}) that $$d(fo, hfo)\le D+2M+4C.$$ However, the   minimal choice of $h$ gives the following $$4D\le d(o, ho)\le d(fo, hfo).$$ Since $D>2M+2C$, we get a contradiction and the \textbf{Case 2} is impossible.

Setting $\tau=2M+8C$ thus completes the verification of the Condition (\textbf{BP}). In a word, we have proved that $\gamma$ is a $(D, \tau)$-admissible path. Moreover, $\gamma$ is a $(D, \tau, L, 0)$-admissible path where $L:=d(\alpha_+, h\alpha_-).$ 

We now choose the constant $D=D(\tau)>0$ by Proposition \ref{saturation}.(2), so there exists a constant $C'=C'(L, D)$ such that $\gamma$ is a $C'$-contracting quasi-geodesic. 

Recall that $\gamma=\bigcup_{n\in \mathbb Z} h^n(\alpha\cdot [\alpha_+, h\alpha_-])$ so $\gamma$ has a finite Hausdorff distance to an orbit of $\langle h\rangle$. Therefore, $h$ is a contracting element: this contradicts to the assumption that $g\in \mathcal {NC}$. The proposition is thus proved.
\end{proof}


We are now ready to complete the proof of Theorem \ref{NCisBarrierfree}.
\begin{proof}[Proof of Theorem \ref{NCisBarrierfree}]
Let $D=D(C)$ be the constant given by Proposition \ref{NCRep} where $C$ is given by Lemma \ref{contractingsegment}. Let $M$ be the constant satisfying Lemma \ref{OArea} so that 
$\mathcal K_{M, D}\subset \mathcal V_{\epsilon, M, f^n}$. Meanwhile,  $\mathcal P_{D, C}\subset \mathcal V_{\epsilon, M, f^n}$ by Lemma \ref{DCArea}.

Every element $g$ in $\mathcal {NC}$ has conjugating representative $h \in G$ in one of the three categories in Proposition \ref{NCRep}. In the first and third cases, the set of such $h$ has been proved to be contained in the barrier-free set $\mathcal V_{\epsilon, M, f^n}$.  Note there are only finitely many possibilities that  $d(o, ho)\le 4D$. Clearly, we can raise $n$ such that $d(o, f^n)$ is sufficiently large so these $h$ lie in $\mathcal V_{\epsilon, M, f^n}$ as well. So the proof is finished.
\end{proof}

\section{Genericity of contracting elements}\label{Section4}
\subsection{Statements and corollaries}

This subsection is to derive various genericity results, including \ref{GenericThm}, from  a more general technical theorem as follows.

Let $\epsilon, M>0$ be given by Theorem  \ref{GrowthTightThm}. Given a  contracting element $f\in G$, denote by $\mathcal {BF}$ the set of elements $g$ in $G$  admitting  conjugacy representatives in a barrier-free set $\mathcal V_{\epsilon, M, f}$ (for simplicity, the constant $M$ will be omited below).

\begin{thm}\label{tightNC}
Assume that the action of $G$ on $\mathrm Y$ is SCC. Then the set $\mathcal {BF}$  is growth-tight: there exists $\varepsilon>0$ such that  
$$ {\sharp \big(N(o, n)\cap \mathcal BF\big)} \le \exp(-\varepsilon n) \cdot {\sharp N(o, n)}$$ for all $n\gg 0$.

If the action is only assumed to be proper, then the set $\mathcal {BF}$ is growth-negligible.
\end{thm}

Therefore, \ref{GenericThm} follows as a direct consequence.
\begin{proof}[Proof  of   \ref{GenericThm}] 
By Theorem \ref{NCisBarrierfree}, the set $\mathcal NC$ of non-contracting elements admits minimal conjugacy representatives in a barrier-free set  $\mathcal V_{\epsilon, M, f}$ for some contracting $f$. Hence,   $\mathcal NC$ is either growth-tight or growth-negligible by the nature of the action, proving \ref{GenericThm}.  
\end{proof}

To give further corollaries of Theorem \ref{tightNC}, we consider a weakly quasi-convex subgroup defined in \cite{YANG10}.    
A subset $X$ in $\mathrm Y$ is called \textit{weakly $M$-quasi-convex} for a constant $M>0$ if for any two points $x,y$ in $X$ there exists a geodesic $\gamma$ between $x$ and $y$ such that $\gamma \subset N_M(X)$. Then a subgroup $\Gamma$ is \textit{weakly quasi-convex} if $\Gamma o$ is weakly  $M$-quasi-convex for some  $M>0$ and$o\in \mathrm Y$.  The following result is proven in \cite[Theorem 4.8]{YANG10}. 
\begin{lem}\label{wqcGTight}
Suppose $G$ admits a proper action on $(\mathrm Y,d)$ with a contracting element. Then  every infinite index weakly quasi-convex subgroup $\Gamma$ of $G$ is contained in a barrier-free set $\mathcal V_{\epsilon, M, f}$ for some contracting element $f$. 
\end{lem}

We now prove that hyperbolic elements are exponentially generic in relatively hyperbolic groups, i.e.: 
\begin{proof}[Proof of Theorem \ref{RelHypGeneric}]
 By   \ref{GenericThm}, it remains to show that all parabolic elements is growth-tight. Since there are finitely many conjugacy classes, we only need to consider one maximal parabolic subgroup $P$ and their conjugates. Since $P$ is quasi-convex (cf. \cite[Lemma 3.3]{GePo2}), we see by Lemma \ref{wqcGTight} that $P$ is contained in a barrier-free set   $\mathcal V_{\epsilon, M, f}$ for some contracting $f$. So Theorem \ref{tightNC} implies that all elements conjugated into $P$ is growth-tight. Therefore, the set of hyperbolic elements is exponentially generic.
\end{proof}  
 
By Lemma \ref{wqcGTight}, the following corollary generalizes the previous result in \cite{YANG10} that convex-cocompact subgroups are growth-tight. 
\begin{lem}\label{FMtight2}
The set of elements conjugated into a fixed convex-cocompact subgroup $\Gamma$ of $G\in \mathbb {Mod}$ is growth-tight. 
\end{lem} 

In light of this result, the last statement of Theorem \ref{ModGeneric2} is proved.

\subsection{Proof of Theorem \ref{tightNC}}

By the definition of $g\in  \mathcal {BF}$,   write $$g=khk^{-1}$$ for some  $k \in G$, where $h$ is an $(\epsilon, f)$-barrier-free element. Denote $\gamma=[o, go]$.

The core of the proof is the following propostion whose proof will be given in next three sections. It says that we can  write the element $g=k' \hat g k'^{-1}$ as an ``almost geodesic'' product, where the elements $k', \hat g$ have the  properties stated in the following.  

\begin{prop}[Almost geodesic form]\label{almostgeodform0}
There exist a constant  $\Delta>0$ and an element $\tilde f \in E(f)$ with the following property. Denote $Z:= E(f)\cup \mathcal V_{\epsilon, \Delta, \tilde f}$.  For each $g\in \mathcal BF$, there exist $k'\in G$ and $\hat  g =k'^{-1}gk'\in Z$ and two points $s, t\in \gamma$ such that   
$$\max\{d(k'\cdot o, s), d(k'\cdot \hat g o,  t)\}\le \Delta.$$
\end{prop}

Assuming Proposition \ref{almostgeodform0}, we complete the proof of  Theorem \ref{tightNC}.

For $\Delta>0$, we consider the annulus $$A(o, n, \Delta)=\{g\in G: |d(o, go)-n|\le\Delta\},$$
and the following holds for any $\Gamma\subset G$,
\begin{equation}\label{criticalexpo}
\limsup_{n\to \infty}\frac{\log \sharp \left(A(o, n, \Delta)\cap \Gamma\right)}{n} = \e \Gamma.
\end{equation}

Denote the distance $l:=d(s, t)$. We thus have
$$|d(o, \hat go) - l |  \le 2\Delta,$$
so $\hat g \in  A(o, l, 2\Delta) \cap Z$ and $$|2d(o, k'o)-n+l|    \le 4\Delta,$$
so $k'\in A(o, (n-l)/2, 2\Delta)$.

Consequently, for each $n\ge 1$,  the number of elements $g\in A(o, n, \Delta) \cap \mathcal BF$   is upper bounded by 
\begin{equation}\label{SUMEQ}
\sum_{1\le l\le n}  \sharp \big(A(o, (n-l)/2, 2\Delta)\big) \cdot \sharp \big(A(o, l, 2\Delta) \cap Z)\big).
\end{equation}
Recall that $E(f)$ is elementary, i.e.: virtually $Z$. 

If the action is SCC, then the set $Z:=\big (E(f)\cup \mathcal V_{\epsilon, \Delta, \tilde f}\big)$ is a growth-tight set with exponent $\omega_1 <\e G$ such that $$\sharp (A(o, l, 2\Delta) \cap Z) \prec \exp(\omega_1 l)$$ for $l>0$.  
Otherwise, if the action is proper, then $Z$ is  growth-negligible: 
$$\frac{\sharp (A(o, l, 2\Delta) \cap Z)}{\exp(\e G l)}\to 0.$$

By definition of $\e G$ in (\ref{criticalexpo}),  there exist a constant  $\omega_2<2\cdot \e G$ such that  $$ \sharp \big(A(o, (n-l)/2, 2\Delta)\big) \prec \exp(\omega_2(n-l)/2).$$ 

\textbf{1).} For SCC actions, the above sum (\ref{SUMEQ}) is upper bounded by
$$
\begin{array}{lll}
\sum_{1\le l\le n}   c \cdot \exp\big(\e G(n-l)/2\big) \cdot \exp(\omega_1 l) \le c\cdot n\cdot \exp(\omega_0 n)  
\end{array}
$$
where $\omega_0:=\max\{\omega_2/2, \omega_1 \}<\e G$. A direct computation shows that $\mathcal BF$ is growth-tight: 
$$\e { A(o, n, \Delta) \cap \mathcal NC} \le \lim_{n\to \infty} \frac{\log\big ( c\cdot n\cdot \exp(\omega_0 n)\big)}{n} \le \omega_0<\e G.$$

\textbf{2).} For proper actions, the above sum (\ref{SUMEQ})  is bounded by
$$
\begin{array}{lll}
&\sum_{1\le l\le n}   c \cdot \exp\big(\e G(n-l)/2\big) \cdot  \sharp \big(A(o, l, 2\Delta) \cap Z)\big)\\
\\
\le & \exp(\omega_2(n-l)/2) \cdot   o(\e G l)  \le   o(\e G n)  
\end{array}
$$
so $\mathcal BF$ is growth-negligible.   

Therefore, the theorem  is proved, modulo Proposition \ref{almostgeodform0} whose proof  will take up the next three sections.

\section{More preliminary: Projection complex and a quasi-tree of spaces}\label{Section5}
The purpose of this section is to   recall a  construction, due to Bestvina, Bromberg and Fujiwara,  of  \textit{projection complex} and correspondingly,  \textit{a quasi-tree of spaces}, a blown-up of it. We assume certain familiarity with their construction and refer to  \cite{BBF} for details.  

In our concrete setting, let us just point out that a  contracting system  $\mathbb X$ with bounded intersection  satisfies the axioms in \cite{BBF} (cf. \cite[Appendix]{YANG7} for this fact). In what follows, we examine their construction and derive a few consequences in this specific setting. 

We first introduce a notion of interval in $\mathbb X$ (cf. \cite[Theorem 2.3.G]{BBF}).   For $K>0$, denote $\mathbb X_K(Y, Z)$ by the set of $W\in \mathbb X\setminus \{Y, Z\}$ such that $\proj_W(Y, Z) > K$. For two points $y, z\in \cup \mathbb X$, the set $\mathbb X_K(y, z)$ is defined as the collection of $W\in \mathbb X$ such that $\proj_W(y, z) > K$. 

\paragraph{\textbf{Projection complex}} The \textit{projection complex} $\mathcal P_K(\mathbb X)$ is a graph such that the vertex set is $\mathbb X$, and two vertices $Y\ne Z\in \mathbb X$ is \textit{adjacent} if $\mathbb X_K(Y, Z)=\emptyset$.

In fact, their construction requires a slightly variant, $d_X(Y, Z)$, of the distance-like function $\proj_X(Y, Z)$. However,   this does not matter in the interest of the present paper since $\proj_X(Y, Z)\sim d_X(Y,Z)$ by \cite[Theorem 2.3.B]{BBF}.

Furthermore, assume that $\mathbb X$ is preserved by a group action of $G$ on $\mathrm Y$. Their fundamental result is then stated as follows.

\begin{thm}\cite[Theorem D]{BBF}\label{projectioncplx}
There exists $K\gg 0$ such that $\mathcal P_K(\mathbb X)$ is a quasi-tree on which $G$ acts co-boundedly.
\end{thm}
 
\paragraph{\textbf{Quasi-tree of spaces}} Following the adjacency in  $\mathcal P_K(\mathbb X)$, a quasi-tree of spaces $\mathbb X$ is constructed   to recover the geometry of each $X\in \mathbb X$ in $\mathcal P_K(\mathbb X)$ where they were condensed to be one point.   

For a constant $N>0$,   a \textit{quasi-tree $\mathcal C_N(\mathbb X)$  of spaces}  is obtained by taking the disjoint union of $\mathbb X$ with   edges of length $N$ connecting each pair of points $(y, z)$ in $(\pi_Y(Z), \pi_{Z}(Y))$ if $\mathbb X_K(Y, Z)=\emptyset$. We denote by  $d_{\mathcal C}$ the induced length metric.

A technical issue is that $X\in \mathbb X$ might not be connected  or even so, the induced metric on $X$ may differ from the one on the ambient space $\mathrm Y$. Since $X$ is quasi-convex, this could be overcome by taking a $C$-neighborhood of $X$ such that any geodesic with endpoints in $X$ lies in $N_C(X)$. It is readily seen that $N_C(X)$ is connected and its induced metric agrees with $d_{\mathrm Y}$ up to a uniform additive error (for instance $4C$). For convenience, each $X\in \mathbb X$ is assumed to be a metric graph by its the Vietoris-Rips complex. See discussion \cite[Section 3.1]{BBF}. 

\begin{thm}\cite[Theorem E]{BBF}\label{quasitreespaces}
For $N\gg K$, $\mathcal C_N(\mathbb X)$ contains $\mathbb X$ as totally geodesic subspaces and for any two $Y, Z\subset \mathcal C_N(\mathbb X)$, the shortest projection of $Y$ to $Z$ in $\mathcal C_N(\mathbb X)$  is uniformly close to the set $\pi_Z(Y)$.
Moreover, if every $X\in \mathbb X$ are uniformly hyperbolic spaces, then  $\mathcal C_N(\mathbb X)$ is hyperbolic. 
\end{thm}

In a hyperbolic space, a totally geodesic subspace is quasiconvex so it is contracting. Recall that the bounded intersection property is equivalent to the bounded projection property in \cite[Lemma 2.3]{YANG6}. Hence, by Theorem \ref{quasitreespaces}, $\mathbb X$ is a contracting system with bounded intersection in $\mathcal C_N(\mathbb X)$.

At last, the following result shall be important in next section.

\begin{lem}\cite[Lemma 3.11]{BBF} \label{standardpath}
There are $\tilde K>K, R > 0$ so that for any $y$ and $z$ in $\mathcal C_N(\mathbb X)$, any geodesic $[y, z]$ passes within $R$-neighborhood of $\pi_X(y)$ and $\pi_X(z)$ for each $X \in \mathbb X_{\tilde K}(y, z)$.
\end{lem}

\section{{Projecting in the quasi-tree of spaces}} \label{Section6}

Let $\mathbb X:=\{g\cdot\ax(f): g\in G\}$ be the collection  of contracting sets with  bounded intersection. Denote by the same $C>0$ the contraction    and   bounded intersection constants for $\mathbb X$. According to Section \ref{Section5}, we consider a projection complex $\mathcal P_K(\mathbb X)$ and its quasi-tree of spaces $\mathcal C_N(\mathbb X)$ endowed with   length metric $d_{\mathcal C}$.

Consider the quadrangle $\square_{g=khk^{-1}}$   by four (oriented) geodesics $\gamma=[o, go]$, $p=[o, ko],  \alpha=[ko, kho]$ and $q=[go, kho]$ as dicpicted in Figure \ref{fig:thm68}. The important relation  $q=g\cdot p$ will be implicit in our discussion.

We now give an overview of this section. After some preliminary observations, we project the quadrangle $\square_{g=khk^{-1}}$ into  $\mathcal C_N(\mathbb X)$ and show that  the top geodesic $\alpha$ becomes uniformly bounded. 

Since the goal of Theorem \ref{tightNC} is to prove that $\mathcal {BF}$ is a growth-tight (resp.  growth-negligible) set,  without loss of generality,   we can assume    $g \notin \mathcal V_{\epsilon, \tilde f}$ for some $\tilde f \in E(f)$: the bottom geodesic $\gamma$ contains an $(\epsilon, \tilde f)$-barrier. The  element  $\tilde f$ will be made sufficiently  ``long'' in a quantitative sense. 

Note that any $(\epsilon, \tilde f)$-barrier in  $\gamma$ gives rise to a ``sufficiently long'' barrier in the left side $p$, for $\tilde f$ is relatively longer. By looking at the projected quadrangle in $\mathcal C_N(\mathbb X)$, the hyperbolicity of $\mathcal C_N(\mathbb X)$ (cf. Lemma \ref{standardpath}) allows to argue that this long barrier in the left side $p$ has to  \textbf{intersect boundedly} with the right side $q$. This is the goal of this section, Lemma \ref{sameLevel}, which provides the base of the further analysis in next Section. 

\subsection{Some auxiliary lemmas}
We begin with an elementary observation facilitating some computations.
\begin{lem}\label{pcapXbound}
For any $X\in \mathbb X$ and any geodesic $p$, the following holds:
$$
\proj_X(p)\le   4C+\diam{N_C(X)\cap p}.
$$
\end{lem}
\begin{proof}
Denote by $p_1$ the part of $p$ before entering into $N_C(X)$, and by $p_2$ the part of $p$ after exiting $N_C(X)$. It is possible that $p_1, p_2$ may be trivial paths.  Hence, the proof is completed by a projection argument as follows:  
$$
\begin{array}{lll}
\proj_{X}(p) &\le \proj_X(p_1\cup p_2)+d((p_1)_+, X)+d((p_2)_-, X)+\diam{N_C(X)\cap p}\\
&\le 4C+\diam{N_C(X)\cap p},
\end{array}
$$
where $\proj_X(p_1\cup p_2) \le 2C$ follows by contracting property, and $$d((p_1)_+, X), d((p_2)_-, X) \le C$$ since $(p_1)_+$ and $(p_2)_-$ both belong to $N_C(X)$. 
\end{proof}

\begin{lem}[Long intersection $\Rightarrow$ Existence of barrier]\label{barrierexists}
There exist  $\epsilon>0$ and $L =L(f)>0$ such that if $\alpha$ is a geodesic satisfying   $$\diam{\alpha \cap N_C(\ax(f))}> L,$$ then it contains an $(\epsilon, f)$-barrier.
\end{lem}
\begin{proof}
Recall that $\gamma:=\ax(f)$ is  a contracting quasi-geodesic.   By hypothesis, we see that $\alpha$ contains a subpath $\bar \alpha$ of length at least $L$ which has finite Hausdorff distance to a subpath $\bar \gamma$ of $\gamma$.  By Assertions (\ref{subpath}) and (\ref{qconvexity}) of Proposition \ref{Contractions}, we have that $\bar \alpha$ is also contracting and thus is $\epsilon$-quasi-convex for some $\epsilon=\epsilon(C)$. A priori,   we can choose $L$ large enough such that $\bar \gamma$ contains a subsegment  labeled by $f$. From the $\epsilon$-quasi-convexity of $\bar \alpha$, this subsegment  stays in the $\epsilon$-neighborhood of $\bar \alpha$, so produces an  $(\epsilon, f)$-barrier as required.
\end{proof}

Let $L>0$ be the constant supplied by Lemma \ref{barrierexists}.

\begin{lem}[]\label{rProjectionEQLem}
Let $\alpha$ be a geodesic side in the quadrangle $\square_{g=khk^{-1}}$. Then  
$$
\proj_{X}(\alpha)\le  4C+L
$$
for any $X\in \mathbb X$. 
\end{lem}
\begin{proof}
By Lemma \ref{pcapXbound}, it suffices to prove that $\diam{N_C(X)\cap \alpha}\le L.$ If the inequality does not hold,  then $\alpha=k\cdot [o, ho]$ contains an $(\epsilon, f)$-barrier by Lemma \ref{barrierexists}: this  contradicts to the \textbf{first reduction} of $h$ that $[o, ho]$ is $(\epsilon, f)$-barrier-free.  The proof is thereby complete.  
\end{proof}

 The point of the next lemma is to determine the appropriate constants $K, N$ for these spaces. 

\begin{lem}\label{NCShort}
There exists a constant $N>0$ with the following property.

Let $\alpha$ be an $(\epsilon, f)$-barrier-free geodesic between $Y$ and $Z$ in $\mathbb X$. Then the endpoints $\alpha_-, \alpha_+$ of $\alpha$ is uniformly bounded in $\mathcal C_N(\mathbb X)$ as follows: $$d_{\mathcal C}(\alpha_-, \alpha_+)\le 2L+N.$$
\end{lem}
\begin{proof}
 
Consider the geodesic $\bar \alpha=[\alpha_-, \alpha_+]$ in $\mathcal C_N(\mathbb X)$ so the goal is to   estimate the length of $\bar \alpha$.

We first consider  the projection complex $\mathcal P_K(\mathbb X)$ where  the constant  $K>L+2C$ is given by Theorem \ref{projectioncplx}.  Observe that $\mathbb X_K(Y, Z)$ is empty. Indeed, since $\alpha$ is $(\epsilon, f)$-barrier-free,  it follows by Lemma \ref{barrierexists} that $\diam{\alpha\cap N_C(X)}<L$ for any $X\in \mathbb X$. Note that   $\proj_X(Y), \proj_X(Z)\le C$. Then $\proj_X(Y, Z)\le \diam{\alpha\cap N_C(X)}+ \proj_X(Y)+\proj_X(Z)\le  L+2C$. Hence, $\mathbb X_K(Y, Z)=\emptyset$ so $Y, Z$ are adjacent in   $\mathcal P_K(\mathbb X)$. 

We then construct a quasi-tree of spaces $\mathcal C_N(\mathbb X)$ where  $N=N(K)$ is given by Theorem \ref{quasitreespaces}. Since $Y, Z$ are adjacent in   $\mathcal P_K(\mathbb X)$, by construction, the subsets $\pi_Y(Z), \pi_Z(Y)$  in $\mathcal C_N(\mathbb X)$ are connected  by edges of length $N$.

By a similar reasoning, we obtain $d(\alpha_-, \proj_Y(\alpha_+)),\; d(\alpha_+, \proj_Z(\alpha_-))\le L.$  Recalling in Section \ref{Section5}, we make the assumption that the induced metric on $Y, Z$ is   identical to the ambiant metric on $\mathrm Y$, up to a uniform error (for simplicity we ignore it here). Moreover, since $Y, Z$ are isometrically embeded in $\mathcal C_N(\mathbb X)$, there exist two paths $\alpha_1, \alpha_2$ of length at most $L$ respectively in $Y, Z$ such that $\alpha_1$ connects $\alpha_-$ and $\proj_Y(\alpha_+)$, and $\alpha_2$ connects $\alpha_+$ and $\proj_Z(\alpha_-)$.

Moreover, the endpoints   of $\bar \alpha$ are connected by a path composing $\alpha_i$ with such an edge of length $N$, hence $d_{\mathcal C}(ko, kho)\le 2L+N.$ The proof is then complete.
\end{proof}

\subsection{Constants}

Let $R>0$ be a   constant   given by Lemma \ref{standardpath}. Since $\mathbb X$ has bounded intersection in $\mathcal C_N(\mathbb X)$,  there exists  $D=D(R)$ such that $$\forall X\ne X' \in \mathbb X: \; \textbf{diam}_{\mathcal C}\big(N_R(X)\cap N_R(X')\big)\le D$$ where the diameter $\textbf{diam}_{\mathcal C}$ is computed using the metric $d_{\mathcal C}$.

We also choose the following constant 
\begin{equation}\label{TildeKEQ}
\tilde K>4C+2R+3L+N+D.
\end{equation}
 satsifying the conclusion of Lemma \ref{standardpath}, a constant $A$ as follows
\begin{equation}\label{Avalue}
A=L+\tilde K+116C.
\end{equation}

Choose a ``long" element $\tilde f\in E(f)$ such that 
\begin{equation}\label{LongfEQ}
d(o, \tilde f o)> 2(25C+A+L+\tilde K+\epsilon),
\end{equation}

By Theorem \ref{GrowthTightThm}, the set $\mathcal V_{\epsilon, \tilde f}$ is growth-tight (resp.  growth-negligible). Since the goal is to prove that $\mathcal {BF}$ is a growth-tight (resp.  growth-negligible) set,  without loss of generality, that we can assume    $g \notin \mathcal V_{\epsilon, \tilde f}$. By definition \ref{barriers}, there exists an element $b\in G$ such that 
$$\max\{d(b\cdot o, \gamma), \; d(b\cdot \tilde fo, \gamma)\}\le \epsilon,$$
whence by (\ref{LongfEQ}) this  gives 
\begin{equation}\label{Xgamma}
\diam{N_{C}(X)\cap\gamma} \ge  2(25C+A+L+\tilde K),
\end{equation}
where $X=b\cdot \ax(f)$. 

By abuse of language, we will say hereafter  that \textit{the element $b$} is an $(\epsilon, \tilde f)$-barrier  of $\gamma$.

\begin{lem}\label{pqintersect}
For each $(\epsilon, \tilde f)$-barrier  $b$ of $\gamma$, the following holds 
$$
\diam{N_C(  X)\cap   p}+\diam{N_C(  X)\cap  q}>2A
$$
where $X=b\cdot \ax(f)$. 
\end{lem}
\begin{proof} 
Let   $x, y$  denote the entry and exit points of $\gamma$ in $N_C(X)$ respectively. If we had
$$
\diam{N_C(X)\cap p}+\diam{N_C(X)\cap q}\le 2A,
$$
then by Lemma  \ref{rProjectionEQLem}, we would obtain  the following:
$$
\begin{array}{lll}
d(x, y)&\le d(x, X)+\proj_X(p) +\proj_X(\alpha)+ \proj_X(q)+d(y, X)\\
&\le 14C+\diam{N_C(X)\cap p}+\diam{N_C(X)\cap \alpha}+\diam{N_C(X)\cap q}\\
&\le 18C+2A+L,
\end{array}
$$
where   $d(x, X), d(y, X)\le C$ for $x, y\in N_C(X)$. This  results a contradiction with (\ref{Xgamma}),   so the lemma is proved.  
\end{proof}

\subsection{Bounded intersection in quadrangle}


\begin{lem}[]\label{sameLevel}
Let $X\in \mathbb X$   such that   $g  X\ne   X$ and $\diam{N_C(X)\cap p}>A$.  Then    $\diam{N_C(X)\cap q}\le  \tilde K$ and $\diam{N_C(X)\cap \gamma}>100C.$ 
\end{lem}

\begin{proof}
The idea of proof is to project the  quadrangle $\square_{g=khk^{-1}}$ to a quadrangle in $\mathcal C_N(\mathbb X)$ with the corresponding geodesics $\bar \gamma=[o, go]$, $\bar p=[o, ko],  \bar \alpha=[ko, kho]$ and $\bar q=[go, kho]$. 

Suppose by way  of contradiciton that $\diam{N_C(X)\cap q}>  \tilde K$ so $X\in \mathbb X_{\tilde K}(go, kho)$ by definition. For any $X\in \mathbb X_{\tilde K}(o, ko)$, let $v, w$ denote the corresponding exit points of $\bar p$ and $\bar q$ in $N_R(X)$, where the constant $R>0$ is given by Lemma \ref{standardpath}. Hence,   
$$d_{\mathcal C} (v, \pi_X(ko)), \; d_{\mathcal C} (w, \pi_X(kho)) \le R.$$ 
 
By Theorem \ref{quasitreespaces},  the subset $X$ is totally geodesic   in $\mathcal C_N(\mathbb X)$ so $$d_{\mathcal C}(\pi_X(ko),  \pi_X(kho))=\proj_X(\pi_X(ko),  \pi_X(kho))$$ 
where the right-hand side is the projection distance  measured in $\mathrm Y$. Noting also that  
$$\proj_X(\pi_X(ko),  \pi_X(kho))\le \proj_X(\alpha)\le 4C+L$$ 
where the second inequality follows by Lemma \ref{rProjectionEQLem}. As a consequence of the above three estimates, we obtain the following 
$$
\begin{array}{ll}
d_{\mathcal C}(v, w)&\le d_{\mathcal C}(v, \pi_X(ko)) +d_{\mathcal C}(\pi_X(ko),  \pi_X(kho))+ d_{\mathcal C}(w, \pi_X(kho)) \\
&\le 4C+2R+L.
\end{array}
$$
 
One needs $d_{\mathcal C}(ko, kho)\le N+2L$ by Lemma \ref{NCShort}.   
Therefore, 
\begin{equation}\label{vwlendiffEQ}
\begin{array}{ll}
&|\len_{\mathcal C}([v, ko]_{\bar p})-\len_{\mathcal C}([w, kho]_{\bar q})|\\
\le&d_{\mathcal C}(v, w)+d_{\mathcal C}(ko, kho)\\
\le & 4C+2R+3L+N,
\end{array}
\end{equation}
where   $\len_{\mathcal C}(\cdot)$ stands for  the length of a path in $\mathcal C_N(\mathbb X)$.

On the other hand, $gX\in \mathbb X_{\tilde K}(go, kho)$, and $gv$ is the exit point of $\bar p$ in $N_R(gX)$ so 
\begin{equation}\label{vgvlenEQ}
\len_{\mathcal C}([v, ko]_{\bar p})=\len_{\mathcal C}([gv, kho]_{\bar q}).
\end{equation}

Recall that $\mathbb X$ has bounded intersection in $\mathcal C_N(\mathbb X)$ so for $gX\ne X \in \mathbb X$,  $$\textbf{diam}_{\mathcal C}({N_R(X)\cap N_R(gX)})\le D.$$
  
Since $q=gp$ then $\bar q=g\bar p$, we obtain that $\textbf{diam}_{\mathcal C}(N_R(X)\cap\bar p) = \textbf{diam}_{\mathcal C}(N_R(gX)\cap\bar q) > \tilde K$.   Consequently, 
$$
\begin{array}{ll}
&|\len_{\mathcal C}([w, kho]_{\bar q})-\len_{\mathcal C}([gv, kho]_{\bar q})|\\
 
>&\displaystyle{\min\{\textbf{diam}_{\mathcal C}(N_R(X)\cap\bar q),\; \textbf{diam}_{\mathcal C}(N_R(gX)\cap\bar q) \}-D}\\
 
>& \tilde K-D.
\end{array}
$$ 
Via (\ref{vgvlenEQ}), this yields a contradiction to (\ref{vwlendiffEQ})    since it was assumed in (\ref{TildeKEQ}) that 
$$\tilde K>4C+2R+3L+N+D.$$
Hence, $\diam{q\cap N_C(X)}\le \tilde K$ is proved.  

Lastly, let us prove that $\diam{N_C(X)\cap \gamma}>100C$. If not, then $\proj_X(\gamma)\le 104C$ by Lemma \ref{pcapXbound}. Let $x, y$ be the entry and exit points of $p$ in $N_C(X)$ respectively so $$\proj_X(p\setminus [x, y]_p)\le 2C$$ by contracting property.  Since $\proj_X(q)\le 4C+\tilde K$ follows by Lemma \ref{pcapXbound} and  $\proj_X(\alpha)\le 4C+L$ by Lemma \ref{rProjectionEQLem}, we see by a  projection argument that 
$$
\begin{array}{lll}
\diam{N_C(X)\cap p} &=   d(x, y)\\
& \le d(x, X)+\proj_X(\alpha\cup \gamma)+\proj_X(q)+\proj_X(p\setminus [x, y]_p)+d(w, X)\\
& \le L+\tilde K+116C  <A,
\end{array}
$$
which is a contradiction. The lemma is thus proved.
\end{proof}

\section{{Almost geodesic decomposition}} \label{Section7}

We are now ready to prove Proposition \ref{almostgeodform0}, the last ingredient in the proof of Thereom \ref{tightNC}.  

Lets first outline the proof:  
The consequence of    Lemma \ref{pqintersect} in the previous section  provides   a contracting set $X\in \mathbb X$ which has a large $A$-intersection with $p$, but  intersect $q$ in a bounded amount by $\tilde K$.   The constant $A$ (\ref{Avalue}) is chosen sufficiently large relative to $\tilde K$ (\ref{TildeKEQ}).

We then focus on the collection of such $X$ with this property, of which the \textit{last} one  intersecting $p$  is given a particular focus on. For the relation $q=gp$, $gX$ is also the last for $q$. In Proposition \ref{almostgeodform},   we establish, case by case, that the intersection of $\gamma$ with the  pair of $(X, gX)$  provides an almost geodesic product of $g$.

In the sequel,  the following  fact is frequently used, whose proof is straightforward by the contracting property and left to the reader.
\begin{lem}[Fellow entry/exit]\label{commonentry}
Let $X$ be a $C$-contracting subset in $\mathrm Y$. Consider two geodesics $\alpha, \beta$ issuing from the same point and both intersecting $N_C(X)$. Then their corresponding entry points of $\alpha, \beta$ in $N_C(X)$ have a distance  at most $4C$. 
\end{lem}

Let us repeat Proposition \ref{almostgeodform0} with some additional quantifiers.

\begin{prop}[Almost geodesic form]\label{almostgeodform}
There exists  $\Delta=\Delta(C, L)>0$ with the following property. Denote $Z:= E(f)\cup \mathcal V_{\epsilon, \Delta, \tilde f}$.  For each $g\in \mathcal BF$, there exist $k'\in G$ and $\hat  g =k'^{-1}gk'\in Z$ and two points $s, t\in \gamma$ such that   
$$\max\{d(k'\cdot o, s), d(k'\cdot \hat g o,  t)\}\le \Delta.$$
\end{prop}


 

\begin{proof}
By   Lemma \ref{sameLevel}, we consider the set of $X\in \mathbb X$ such that $\diam{N_C(X)\cap p}>A$ and $\diam{N_C(X)\cap q}\le \tilde K$.  The proof shall treat two mutually exclusive configurations.

\textbf{Configuration I.} Assume that there exists a contracting set $X$ such that   $gX=X$. Let $x, y$ be the corresponding entry and exit points of $\gamma$ in $N_C(X)$. 

If $z, w$ denote the entry and exit points of $p$ in $N_C(X)$ respectively then so do $gz, gw$ for $q$ in $N_C(gX)$. By Lemma \ref{commonentry}, we have $d(z, x)\le 4C$ and $d(gz, y)\le 4C$, which implies  $$d(gx, y)\le 8C.$$

For concreteness, assume that $X=b\ax(f)$ for some $b\in G$. In addition, let $M$ be diameter of fundamental domain of the action of $E(f)$ on $\ax(f)$, so there exists $k'\in bE(f)$ such that $$d(k'o, x) \le M+C$$ and 
$$
\begin{array}{lll}
d(gk'o, y)&\le d(gk'o, gx)+d(gx, y)\\
&\le M+9C.
\end{array}
$$    

Since $b\cdot\ax(f)=gb\cdot \ax(f)$, it follows by   definition of $E(f)$ that $g\in bE(f)b^{-1}$. 
Thus, the element $\hat g:=k'^{-1}gk'$ lies in $E(f)\subset Z$ for the above $k'\in bE(f)$. 

Setting $\Delta:=M+5C$, $s=x$ and $t=y$ the desired points on $\gamma$,  the proof in Configuration I is finished.
\begin{figure}[htb] 
\centering \scalebox{0.8}{
\includegraphics{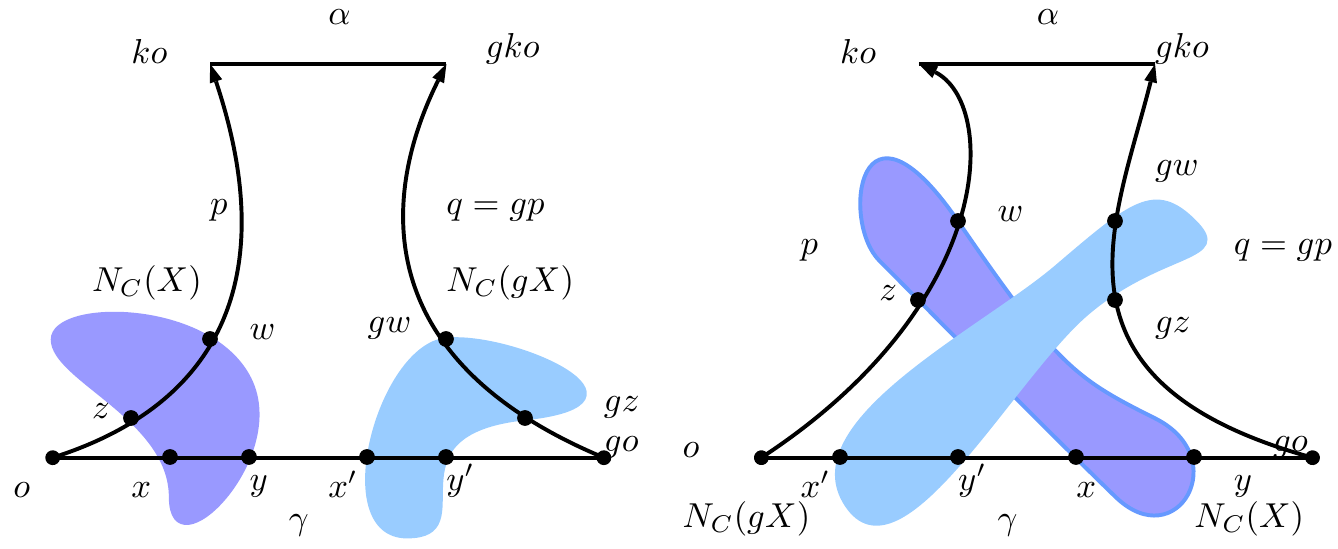} 
} \caption{Case 1 (left) and Case 2 (right)  in Proposition \ref{almostgeodform}}  \label{fig:thm68}
\end{figure}

\textbf{Configuration II.}  Let $X$ be a \textit{last} contracting set $X\in \mathbb X$  for $p$ with the following defining property: $$\diam{N_C(X)\cap p}> A$$ and $$ \diam{N_C(X)\cap [w, p_+]_p}\le A,$$ where $w$ is the exit point of $p$ in $N_C(X)$. By   the symmetry of $q=gp$, it then follows that $gX$ is the last for $q$ as well. Keep in mind that $gX\ne X$ through out the discussion of this Configuration. 

Denoting by $x', y'$ the entry and exit points repsectively in $N_C(gX)$, we are led to deal with the following two cases.

\textbf{Case 1.} $[x, y]_\gamma$ appears before $[x', y']_\gamma$. Noting   $\diam{N_C(X)\cap q}\le \tilde K$, we obtain $\proj_X(q)\le \tilde K+4C$ from Lemma \ref{pcapXbound}. Together, $\proj_X(\alpha)\le L+4C$ by Lemma \ref{rProjectionEQLem},  we see that
\begin{equation}\label{wydistEQ}
\begin{array}{ll}
d(w, y)&\le d(w, X)+ {\proj_X([w, p_+]_p)} +  {\proj_X(\alpha\cup q)}+ {\proj_X([y, \gamma_-]_\gamma})+d(y, X)\\
&\le 12C+L+\tilde K,
\end{array}
\end{equation}
where  ${\proj_X([w, p_+]_p)},  {\proj_X([y, \gamma_-]_\gamma})\le C$ by contracting property, and $d(w, X), d(y, X)\le C$. 

The same argument shows
\begin{equation}\label{gwxdistEQ}
d(gw, x')\le 12C+L+\tilde K.
\end{equation}

Combining (\ref{wydistEQ}) and (\ref{gwxdistEQ}), we thus obtain 
$$d(gy, x')\le d(gy, gw)+d(gw, x')\le 2(12C+L+\tilde K).$$

Again, the cocompact action of $E(f)$ on $\ax(f)$ provides an element $k'\in bE(f)$ such that 
$$
d(k'o, y) \le C+M
$$
and
$$
\begin{array}{lll}
d(gk'o, x')&\le d(gk'o, gy)+d(gy, x')\\
& \le 2(12C+L+\tilde K)+M+C.
\end{array}
$$

Denote $s=y, t=x'$. Setting $\hat g:=k'^{-1}gk'$ and $\Delta=2(13C+L+\tilde K)+M$,  it suffices to establish the following claim which then implies $\hat g\in \mathcal V_{\epsilon, \Delta, \tilde f}$, completing the proof of the Case 1.  
\begin{claim}
$[y, x']_\gamma$ is $(\epsilon, \tilde f)$-barrier-free.
\end{claim} 
\begin{proof}[Proof of the claim]
If not, then $[y, x']_\gamma$ contains an $(\epsilon, \tilde f)$-barrier $b_\star$  so for $X_\star:=b_\star\ax(f)$, we have 
\begin{equation}\label{uvcomplprojEQ}
\proj_{X_\star}(\gamma\setminus[u, v]_\gamma)\le 2C
\end{equation} by the contracting property, where $u, v$ are the entry and exit points  of $\gamma$ respectively in $N_C(X_\star)$. 

Denoting $p_1:=[w, p_+]_p, q_1:=[q_+,gw]_q$,  we deduce $$\max\{\diam{p_1\cap N_C(X_\star)}, \diam{q_1\cap N_C(X_\star)}\}\le A,$$ from the assumption that $X$ and $gX$ are last. This implies, by  Lemma \ref{pcapXbound}, that
\begin{equation}\label{p12projEQ}
\proj_X(p_1\cup p_2) \le 2A+8C.
\end{equation}

On the other hand,   applying Proposition \ref{Contractions}.5, we obtain from  (\ref{wydistEQ}) and (\ref{gwxdistEQ}) that  
\begin{equation}\label{twoprojEQ}
\max\{\proj_{X_\star}([w, y]), \;\proj_{X_\star}([gw, x'])\}\le 13C+L+\tilde K.
\end{equation}

Finally, consider the hexagon formed by $[y, w], p_1, \alpha, q_1,[gw, x']$ and $[y,x']_\gamma$. Recall that  $\proj_{X_\star}(\alpha)\le 4C+L$  by Lemma \ref{rProjectionEQLem}. Using the above estimates from (\ref{uvcomplprojEQ}), (\ref{p12projEQ}) and  (\ref{twoprojEQ}), we obtain
$$
\begin{array}{ll}
\diam{N_C(X_\star)\cap \gamma}&\le d(u, X_\star)+ \proj_{X_\star}(p_1\cup \alpha\cup q_1)+\proj_{X_\star}(\gamma\setminus[u, v]_\gamma)\\
&\;\;+\proj_{X_\star}([w, y])+\proj_{X_\star}([gw, x'])+d(v, X_\star)\\
&\le 50C+2A+2L+2\tilde K,
\end{array}
$$
a contradiction with the inequality (\ref{Xgamma}). The claim is thus proved. 
\end{proof}

\textbf{Case 2.} $[x', y']_\gamma$ appears before $[x, y]_\gamma$. We obtain $d(z, x)\le 4C$ and $d(gz, y')\le 4C$  by   Lemma \ref{commonentry} so 
\begin{equation} \label{gxydistEQ}
d(gx, y')\le 8C
\end{equation} 
which gives by Proposition \ref{Contractions}.5: \begin{equation} \label{gxyprojEQ}
\proj_X([gx, y'])\le 9C.
\end{equation}   

Let us look at two geodesics $[y', go]_\gamma$ and $[gx, go]_\gamma$ with at the same terminal point but their initial points   $8C$-apart. First of all, observe that $${N_C(X)\cap [gx, go]}\ne \emptyset.$$ Indeed, if not, we would obtain from the contracting property that $$\proj_X([gx, go])\le C,\; \proj_X([y', go]_\gamma\setminus [x, y]_\gamma)\le2C$$ and hence 
$$
\begin{array}{ll}
\diam{N_C(X)\cap \gamma} &\le \proj_X([gx, go])+\proj_X([gx, y'])+\proj_X([y', go]_\gamma\setminus [x, y]_\gamma)\\
&\;\; +d(x, X)+d(y, X)\\
&\le 14C,
\end{array}
$$
where $\proj_X([gx, y'])\le 9C$ by (\ref{gxyprojEQ}) and $d(x, X), d(y, X) \le C$.  However,  this contradicts to the conclusion of Lemma \ref{sameLevel}. Hence, ${N_C(X)\cap [gx, go]_\gamma}\ne \emptyset.$

So, consider  the entry point $u$ of $[gx, go]_\gamma$ in $N_C(X)$. We next need to determine the points $s, t$ on $\gamma$. By a projection argument, it follows that: 
$$
\begin{array}{ll}
d(x, u)&\le d(x, X)+\proj_X([y',x])+\proj_X([gx, y'])+\proj_X([gx, u])+d(u, X)\\
&\le 13C,
\end{array}
$$
where (\ref{gxyprojEQ}) is used and $\proj_X([gx, u]), \proj_X([y',x])\le C$ by contracting property. Since $d(gx, y')\le 8C$ in (\ref{gxydistEQ}) we have
$$d(gy', u)=|d(gx, u)-d(gx, gy')|\le d(gx, y')+d(x, u)\le 21C,$$ which in turn shows $$d(gy', x)\le d(gy', u)+d(u,x)\le  34C.$$
Denoting $s=y', t=x$ gives our desired points. As in Case (1), it remains to prove the following.   
\begin{claim}
$[y', x]_\gamma$ is  $(\epsilon, \tilde f)$-barrier-free as well.
\end{claim}
\begin{proof}[Proof of the claim]
Indeed, suppose to the contary that $b_\star$ is a barrier of $[y', x]_\gamma$ so for $X_\star:=b_\star\ax(f)$, the   inequality (\ref{Xgamma}) implies
\begin{equation} \label{XingammaEQ}
\diam{N_C(X_\star)\cap \gamma}>A+13C.
\end{equation}  Since $d(z,x) \le 4C$ by Lemma \ref{commonentry}, we get   $\proj_{X_\star}([z, x])\le 5C$ by Proposition \ref{Contractions}.5. 

Denote $\beta = [p_-, x]_p$. We   look at the triangle at left corner (cf. Figure \ref{fig:thm68}). We shall prove that $\diam{N_C(X_\star)\cap \beta}> A$. If  $\diam{N_C(X_\star)\cap\beta}\le A$ so $\proj_{X_\star}(\beta )\le A+4C$ by Lemma \ref{pcapXbound}, then 
 we obtain $$
 \begin{array}{ll}
\diam{N_C(X_\star)\cap \gamma}&\le \proj_{X_\star}(\beta)+\proj_{X_\star}([z, x])+ \proj_{X_\star}(\beta\setminus    N_C(X_\star))+2C\\
&\le A+13C
\end{array}
$$
where $\proj_{X_\star}(\beta\setminus    N_C(X_\star)) \le 2C$ by contracting property.  This contradicts to the above inequality (\ref{XingammaEQ}) so it follows as desired: $$\diam{N_C(X_\star)\cap p}>\diam{N_C(X_\star)\cap  \beta}> A.$$ 

Similarly, we obtain  
$\diam{N_C(X_\star)\cap q}>A>\tilde K$ by looking at   the triangle at right corner (cf. Figure \ref{fig:thm68}).  By Lemma \ref{sameLevel}, we must have $gX_\star= X_\star$. This contradicts to the assumption made in   Configuration II. Therefore, such a barrer $b_\star$ does not exist so the claim is established.
\end{proof}
 
With $s=x', t=x$, we can follow the same line in the Case (1) using the action of $E(f)$ on $\ax(f)$.  So the element $\hat g:=k'^{-1}gk'$ is $(\epsilon,  \Delta, \tilde f)$-barrier-free for $\Delta=M+35C$.  This  concludes the proof of Configuration II and thereby the proposition.
\end{proof}

\bibliographystyle{amsplain}
 \bibliography{bibliography}

\end{document}